\documentclass[11pt, reqno]{amsart}

\evensidemargin
\oddsidemargin 
\makeatletter\@addtoreset{equation}{section}\makeatother

\usepackage{amsmath,amsthm}
\usepackage{aliascnt}

\usepackage[dvips]{color}

\usepackage[pdftex,bookmarks=true,bookmarksnumbered=true,breaklinks=true, colorlinks=true]{hyperref}
\usepackage[capitalise]{cleveref}

\usepackage{color,calc,enumerate}

\makeatletter
        {
          \raggedright
        \setlength{\rightmargin}{\leftmargin}
        \setlength{\itemsep}{-12pt}
        \setlength{\parsep}{20pt}
        \begin{lrbox}{\@tempboxa}%
        \begin{minipage}{\linewidth-2\fboxsep}
        }%
        {
        \end{minipage}%
        \end{lrbox}%
        \fbox{\usebox{\@tempboxa}}\newline
        }%
\makeatother

\usepackage{pstricks,pst-node,pst-text,pst-3d}

\newcommand{\eps}{\varepsilon}

\newcommand{\ssup}[1] {{\scriptscriptstyle{({#1}})}}

\newcommand{\R}{\mathbb R}
\newcommand{\Zz}{\mathcal Z}

\newcommand{\Gg}{\mathcal G}

\newcommand{\NN}{\mathbb N}

\newcommand{\EE}{\mathbb E}

\newcommand{\eE}{\varepsilon}

\renewcommand{\phi}{\varphi}

\newcommand{\Cc}{{\mathcal C}}

\newcommand{\Ee}{\mathcal E}

\newcommand{\DD}{\Delta}

\newcommand{\e}{\mathrm{e}}

\newcommand{\di}{\;\textrm{d}}

\newcommand{\N}{\mathbb N}
\newcommand{\PP}{\mathbb P}

\theoremstyle{plain}

\newtheorem{theorem}{Theorem} \crefname{theorem}{Theorem}{Theorems}

\newtheorem{rmk}{Remark}[section]
	\crefname{rmk}{Remark}{Remarks}
\newtheorem*{nrmk}{Remark on notation}
\newtheorem*{ack}{Acknowledgement}

\newaliascnt{lem}{rmk}
	\newtheorem{lem}[lem]{Lemma}
	\aliascntresetthe{lem}
	\crefname{lem}{Lemma}{Lemmas}

\newaliascnt{cor}{rmk}
	
	\aliascntresetthe{cor}
	\crefname{cor}{Corollary}{Corollaries}

\newaliascnt{dfn}{rmk}
	
	\aliascntresetthe{dfn}
	\crefname{dfn}{Definition}{Definitions}

\newaliascnt{prop}{rmk}
	\newtheorem{prop}[prop]{Proposition}
	\aliascntresetthe{prop}
	\crefname{prop}{Proposition}{Propositions}

\newcommand{\heap}[2]  {\genfrac{}{}{0pt}{}{#1}{#2}}
\newcommand{\sfrac}[2] {\mbox{$\frac{#1}{#2}$}}

\def\1{{\mathchoice {1\mskip-4mu\mathrm l}      
{1\mskip-4mu\mathrm l}
{1\mskip-4.5mu\mathrm l} {1\mskip-5mu\mathrm l}}}

\renewcommand{\ssup}[1] {{\scriptscriptstyle{({#1}})}}

\newcommand{\con}[1] {\overset{#1}{\leftrightarrow}}

\begin{document}


\title[Distances in critical random networks]
{Distances in scale free networks at criticality}

\author[Steffen Dereich, Christian M\"onch, and Peter M\"orters]{Steffen Dereich, Christian M\"onch, and Peter M\"orters}
\maketitle

\vspace{-0.3cm}

\begin{quote}
{\small {\bf Abstract:} }
Scale-free networks with moderate edge 
dependence experience a phase transition between ultrasmall and small world behaviour when the power law exponent passes the critical value of three. Moreover, there are laws of large numbers for the graph distance of two randomly chosen vertices in the giant component. When the degree distribution follows a pure power law these show the same asymptotic distances of \smash{$\frac{\log N}{\log\log N}$} at the critical value three, but  in the ultrasmall regime reveal a difference of a factor two between the most-studied rank-one and preferential attachment model classes. In this paper we identify the critical window where this factor emerges. We  look at models from both classes when the asymptotic proportion of vertices with degree at least~$k$ scales like\smash{ $k^{-2} (\log k)^{2\alpha + o(1)}$} and show that for preferential attachment networks the typical distance is \smash{$\big(\frac{1}{1+\alpha}+o(1)\big)\frac{\log N}{\log\log N}$} in probability as the number~$N$ of vertices goes to infinity. By contrast the typical distance in  a rank one model with the same asymptotic degree sequence is \smash{$\big(\frac{1}{1+2\alpha}+o(1)\big)\frac{\log N}{\log\log N}.$} As $\alpha\to\infty$ we see the emergence of a factor two between the length of shortest paths as we approach the ultrasmall regime. 
\end{quote}

\vspace{0.5cm}

{\footnotesize
\vspace{0.1cm}
\noindent\emph{MSc Classification:}  Primary 05C82 
Secondary 05C80, 
60C05, 
90B15. 

}

\vspace{0.5cm}

\section{Background and Motivation}

\noindent
Scale-free networks are characterised by the fact that, as the network size goes to infinity, the asymptotic 
proportion of nodes with degree at least~$k$ behaves like $k^{-\tau+o(1)}$ for some power law exponent~$\tau$. There are a number of mathematical models for scale-free networks,  in the class of {\bf rank-one models} the probability that
two vertices are directly connected is asymptotically equivalent to the product of suitably defined weights $w_{v}$ associated to the 
vertices $v$  in a network $\Gg_N$ with vertex set  $[N]:=\{1,\ldots,N\}.$
Examples of rank-one models are the \emph{Chung-Lu model} where
$$\PP(u\leftrightarrow v)=\frac{w_u w_v}{\sum_{i=1}^Nw_i}\wedge 1,\qquad \mbox{ for } u,v\in[N],$$ 
the \emph{Norros-Reittu model} in which \begin{equation}\label{eq:NRprob}\PP(u\leftrightarrow v)= 1-\e^{-\frac{w_u w_v}{\sum_{i=1}^N w_i}},\qquad \mbox{ for } u,v\in[N],\end{equation} where $(w_i)_{i=1}^N$ is a deterministic
or random sequence of weights,  and the \emph{configuration model} in which each vertex is assigned  a degree chosen randomly from a given degree 
distribution and the weights are the degrees themselves. 
\medskip

A popular alternative to rank one models are the {\bf preferential attachment models} introduced by Barab{\'a}si and Albert. The original 
\emph{Barab{\'a}si-Albert model} (see Bollobas et al.~\cite{bollobs_degree_2001} for a rigorous definition) is a dynamical network model in 
which new vertices connect to a fixed number of existing vertices with a probability proportional to their degree. In this model the power 
law exponent is always $\tau=3$. Recent variants introduced by van der 
Hofstad et al.~\cite{dommers_diameters_2010} and Dereich and 
M\"orters~\cite{dereich_random_2009}, allow the connection probability to be proportional to a function of the degree and can therefore generate networks with variable power law exponent~$\tau>2$. 
Physicists have predicted that all these models of scale-free networks with the same power law exponent share essentially 
the  same global topology, see for example~\cite{albert_statistical_2002}.
\medskip

\pagebreak[3]

Indeed, all models listed above have been shown to experience a phase transition at power law exponent three. If $\tau>3$ randomly chosen vertices in the largest connected component have a distance of aymptotic order logarithmic in the network size, whereas for $2<\tau<3$ the distance behaves like an iterated logarithm of the network size, this phase is called the \emph{ultrasmall regime}.
\medskip

\pagebreak[3]

At the critical value $\tau=3$ a fine analysis has been performed
by Bollob\'{a}s and Riordan in their seminal paper~\cite{BR04}. They show that two randomly chosen vertices in the original Barab\'asi-Albert model have a graph distance~\smash{$(1+o(1))\,\log N/\log\log N$}. The same result holds for a variety of other models of scale-free networks when the asymptotic proportion of vertices with degree at least~$k$ scales precisely like $k^{-2}$. Examples include the rank one models of 
Chung and Lu, of Norros and Reittu, inhomogeneous random graphs with a suitable choice of kernel, 
and the configuration model.
\medskip

It was therefore believed that distances in preferential attachment models behave similar to 
distances in the configuration model with the same tail of the asymptotic degree distribution, see for example~\cite{hofstad_phase_2007}. It thus came as 
a surprise when a finer analysis in~\cite{dereich_typical_2012} showed that in the ultrasmall regime, i.e.\ when the 
power law exponent is in the range $2<\tau<3$, 
distances in preferential attachment models are twice as long as in the rank one models above when they have the same tail of the degree distribution. This is due to
the fact that two vertices of high degree in the preferential attachment model are much more likely to be connected 
by a path of length two, rather than a single edge as in the rank one models.
\medskip

\pagebreak[3]

It is the aim of the present paper to study the emergence of this factor two at the critical value $\tau=3$.
Does the factor occur at a sharp threshold and if so where is this threshold? 
Or is there a smooth transition between the factors one and two in a suitably chosen critical window?
To answer these questions we need to consider models that can be studied with logarithmic corrections 
in the tails of the aymptotic degree distribution, which requires us to look at preferential attachment models with \emph{nonlinear} attachment rules, an area essentially unexplored in the rigorous literature.
We look at preferential attachment models in the framework of~\cite{dereich_random_2009, dereich_random_2013}. This allows the attachment probabilities to be chosen as concave functions of the vertex degree, giving enough flexibility to generate varying asymptotic degree distributions. The critical window for our study emerges when the asymptotic proportion of nodes with degree at least~$k$ scales like $k^{-2} (\log k)^{2\alpha + o(1)},$ for some $\alpha>0$. We compare our results on preferential attachment networks with those on the Norros-Reittu model with i.i.d. weights whose degree sequence has the same tail behaviour. Our main result shows that typical distances in the preferential attachment networks are bigger by an asymptotic factor of \smash{$(1+2\alpha)/{(1+\alpha)}$}, which converges to two as $\alpha\uparrow\infty$. 

\section{Statement of the main results}

\noindent
Our main result concerns the variant of the preferential attachment model introduced in \cite{dereich_random_2009}, which has the advantage over other variants
of remaining tractable  even when the connection probability is a nonlinear function of the degree of the older vertex. To define the model precisely,
fix a concave function $f \colon \NN_0 \to (0,\infty)$, which is called the \emph{attachment rule}, and define a sequence of random graphs $(\Gg_N)_{N\in\NN}$ in the following way:
\begin{enumerate}
	\item The initial graph $\Gg_1$ is a single vertex labelled $1$.
	\item Given $\Gg_N$, the graph $\Gg_{N+1}$ is obtained by
		\begin{itemize}
			\item adding a new vertex labelled $N+1$;
			\item independently for any vertex with label $m\leq N$ insert an edge between this vertex and the new vertex with probability
			$$\frac{f(\Zz[m,N])}{N}\wedge 1,$$ 
			where $\Zz[m,N]:=\sum_{i=m+1}^N\1\{m \leftrightarrow i\}$ is the number of younger vertices connecting to~$i$ in $\Gg_N$. 
	\end{itemize}
\end{enumerate}\pagebreak[3]
If we orient all edges from the younger to the older vertex we can interpret  $\Zz[m,N]$ as the \emph{indegree} of the vertex labelled~$i$ in 
the oriented graph derived from $\Gg_N$.  Note however that throughout this paper we consider the graphs $\Gg_N$ as unoriented and the notions of connectivity and graph distance~$d_N$ taken in $\Gg_N$ are with reference to unoriented edges. For any potential edge $(v,w)\in [N]^2$ with $v<w$ we write $v\leftrightarrow w$ if we wish to indicate that $(v,w)$ is contained in $\Gg_N$. When it is convenient to stress the original orientation we write $v\leftarrow w$ or $w\rightarrow v$.
\medskip

The following theorem identifies the class of attachment rules which produces typical distances of order ${\log N}/{\log\log N}$. It is the main result of this paper.
\begin{theorem}\label{thm:PA}
Let $(\Gg_N)_{N\in\NN}$ be the sublinear preferential attachment model obtained from a concave attachment rule $f$ satisfying
\begin{equation}\label{eq:fcondition}
f(k)=\frac{1}{2}k+\frac{\alpha}{2}\frac{k}{\log k}+o\Big(\frac{k}{\log k}\Big),
\end{equation}
for some $\alpha> 0.$ Consider two vertices $U,V$ chosen independently and uniformly at random from the largest connected component $\Cc_N\subset\Gg_N$, then $$d_N(U,V)=\Big(\frac{1}{1+\alpha}+o(1)\Big)\frac{\log N}{\log\log N}\ \ \textrm{ with high probability as }N\to\infty.$$
\end{theorem}

{
\noindent
The lower bound in \cref{thm:PA} uses a standard path counting argument and first moment bounds. The upper bound is much more difficult to obtain and we use a rather complicated second moment argument for the size of the neighbourhood of a typical vertex and combine it with a result concerning a dense subgraph among the oldest vertices using sprinkling-type arguments.
}
\bigskip

It is shown in~\cite{dereich_random_2009} that the asymptotic degree distribution in the preferential attachment graph $\Gg_N$ 
with the attachment rule given in~\eqref{eq:fcondition} satisfies
\begin{equation}\label{empi}
\frac1N \sum_{v\in\Gg_N} \1\{ \mbox{degree}(v)\geq k\} =  k^{-2} (\log k)^{2\alpha + o(1)}
\quad \mbox{ in probability.}
\end{equation}

{This can be seen as follows. According to \cite[Theorem 1.1.]{dereich_random_2009}, the asymptotic indegree distribution in $\Gg_N$ is explicitly given by \begin{equation}\label{eq:degdistri}
\mu_k = \frac{1}{1+f(k)}\prod_{j=0}^{k-1}\frac{f(j)}{1+f(j)} \textrm{ for }k\in\N\cup\{0\},
\end{equation}
whereas the outdegree is asymptotically Poisson distributed. Choosing an affine attachment rule $f(k)= \gamma k+\beta$, one obtains from \eqref{eq:degdistri} by use of Stirling's formula,
$$\mu_k = O\big(k^{-1+\frac1\gamma}\big),$$
cf. \cite[Example 1.3]{dereich_random_2009}. This illustrates that the network is a small world for $\gamma<1/2$ and ultrasmall if $\gamma>1/2,$ since for affine $f$ the power law tails of the indegree distribution dominate the exponential tails of the outdegree distribution. Fixing $\gamma=1/2$ and adding a logarithmically decaying perturbation into the linear factor, i.e. $$f(k) = \Big(\frac12 + \frac{\alpha}{2 \log k}\Big) k, \text{ for } k\geq 2,$$ yields, using the Taylor expansion of $\log(\cdot)$,
$$
\log \Big(\frac{f(j)}{1+f(j)}\Big)= -\frac{2}{j}+\frac{2\alpha }{j\log j} + \frac{2}{j^2} + O\Big(\frac{1}{j (\log j)^2}\Big),
$$
for large $j\in\N.$ Note that the latter two terms are summable in $j$ whereas the first two terms are not. Hence, \eqref{eq:degdistri} implies that 
$$\log \mu_k = -3 \log k + 2\alpha\log\log k +O(1), \text{ as } k\to \infty,$$
since $\sum_{j=1}^k j^{-1} \approx \log k$ and $ \sum_{j=2}^k j^{-1}(\log j)^{-2}\approx \log\log k.$ Noting that the left hand side of \eqref{empi} converges to 

$\sum_{j=k}^\infty \mu_j$ 
one obtains the asserted scaling. The same derivation together with a somewhat tedious but straightforward analysis of the lower order terms appearing yields \eqref{empi} for the more general shapes of $f$ given in \eqref{eq:fcondition}.
\medskip

The calculation of the last paragraph also explains our particular choice of attachment rule. At the critical point $\tau=3$ (or $\gamma=1/2$), the scale of the typical distances is rather sensitive to the parameters of the network model under consideration. We limit ourselves in \cref{thm:PA} to those $f$ which change precisely the factor in front of the $\log N/\log \log N$ term obtained in \cite{BR04} to illustrate the emergence of the characteristic factor $2$ that separates distances in preferential attachment models from distances in rank-1-models in the ultrasmall regime. Note that in \cite{BR04} the authors rely on the equivalence of certain instances of the Barab\'asi to another combinatorial model making it very challenging to adapt their arguments to the regime we are interested in.
\medskip

In principle, it is possible to obtain distances on a variety of scales between $\log N$ and $\log\log N$ other than $\log N/\log\log N$ at $\tau=3$. One may be able to reverse engineer the correct attachment function and then give a rigorous proof along the same lines as ours. We have refrained from doing so, since many of our calculations use explicit estimates and are not straightforwardly generalisable. A formula relating the typical distance explicitly to $f$ or to the degree sequence $(\mu_k)_{k\geq 0}$, as it can be given for rank-1-models, see e.g. \cite{chung2003}, seems presently out of reach for nonlinear preferential attachment models.
}\medskip

We contrast the result of Theorem~\ref{thm:PA} on typical distances in the \emph{preferential attachment model} with a result 
on typical distances in the \emph{Norros-Reittu model} with an i.i.d.\ weight sequence parametrised to obtain the same tail 
behaviour of the empirical degree distribution. We choose this model for definiteness but the result extends easily to other 
rank-one models, such as the Chung-Lu model, and to deterministic weight sequences with similar asymptotics.%
\bigskip%

To define the model, given a distribution on 
the positive reals we generate a sequence $W=(W_i)_{i=1}^{\infty}$ 
of i.i.d.\ random variables with this distribution. Let $L_N=
\sum_{n=1}^NW_n$ denote the total weight of the vertices in $[N]$. 
For fixed $N$ and given the weights $W_1,\ldots, W_N$ we construct the random graph $\mathcal{H}_N=\mathcal{H}_N(W)$ with vertex set $[N]$ as follows:
\begin{itemize}
\item Between any two distinct vertices $v,w\in[N]$ the number of edges is Poisson
distributed with parameter $W_vW_w /L_N$, independent of all other edges.
\item Parallel edges are merged to obtain a simple graph.
\end{itemize}
\medskip
\begin{theorem}\label{thm:CM}
Let $(\mathcal{H}_N)_{N\in\N}$ denote the Norros-Reittu model with weight distribution satisfying
\begin{equation}\label{eq:CMcondition}
\PP( W_1\geq k) =k^{-2}(\log k)^{2\alpha+o(1)},\end{equation} 
for $\alpha>0.$ Consider two vertices $U,V$ chosen independently and uniformly at random from the largest connected component $\Cc_N\subset\mathcal{H}$, then $$d_N(U,V)=\Big(\frac{1}{1+2\alpha}+o(1)\Big)\frac{\log N}{\log\log N}\ \ \textrm{ with high probability as }N\to\infty.$$
\end{theorem}
\bigskip

\noindent
We observe that the characteristic difference in the typical distances between preferential attachment models and rank-one models in the ultrasmall
regime does not occur suddenly at the phase transition, but arises gradually in a critical window.  For networks with empirical degree distributions decaying as in~\eqref{empi} there is a factor of \smash{${(1+2\alpha)}{(1+\alpha)}$} between the typical distances  in the two types of networks. This factor converges to two as we approach the ultrasmall regime by letting $\alpha\uparrow\infty$, and converges to one as we approach the linear case by 
letting $\alpha\downarrow 0$. 
%
A heuristic explanation for this transition is that in the preferential attachment model in the critical window the probabilities that two vertices of high indegree are connected directly
or via a young connector vertex are on the same scale. Hence the asymptotical proportion of the transistions between vertices  on a typical short path that use a connector,
is a constant strictly between zero and one. This constant turns out to be 
${\alpha}/({1+\alpha})$ and this yields a factor $1+{\alpha}/({1+\alpha})$ by which
the length of shortest paths in the preferential attachment model exceed that in the rank-one models.
\medskip

Qualitatively different behaviour for the preferential attachment and rank-one model class can also be observed when studying robustness of the giant component
under targeted attack, see Eckhoff and M\"orters~\cite{Maren}, or in the behaviour of the size of the giant component near criticality, see forthcoming work of Eckhoff, M\"orters and Ortgiese~\cite{EMO16}.

\section{Proof of lower bounds -- preferential attachment}

Lower bounds for average distances are proved using a first moment method. To set it up, Section~3.1 provides bounds for expected degrees in the preferential attachment model, which are used in Section~3.2 to prove the lower bound in \cref{thm:PA}.

\begin{nrmk}
 In all subsequent sections a subscript number on a constant refers to the place where it is defined, e.g. $C_{1.23}$ is the constant introduced in Lemma 1.23., $C_{(1.24)}$ the same constant as in equation $(1.24)$, etc.
\end{nrmk}
\subsection{Degree asymptotics for preferential attachment}\label{sec:moments}

It follows immediately from the definition of the preferential attachment graph, that the network is entirely represented by the collection $(\Zz[1,n])_{n\geq 1},(\Zz[2,n])_{n\geq 2},\dots$ of independent Markov chains, which we refer to as \emph{degree evolutions}. In this section, we derive lower and upper bounds for $\EE f(\Zz[m,n]).$ For conciseness in the formulation of later results, we allow $(\Zz[m,n])_{n\geq m}$ to start in any integer $k\in\NN$ and denote the resulting distribution 
by~$\PP^k$, its expectation by $\EE^k$.
\begin{lem}\label[lem]{lem:degmart2}
Let $k,m\in\NN$, $\Zz[m,m]=k$ be fixed and define, for $n\geq m$,
\begin{equation*}
X(n)=\frac{f(\Zz[m,n])}{\xi(m,n)}\ \ \text{and}\ \ Y(n)=\frac{f(\Zz[m,n])^2+\frac12 f(\Zz[m,n])}{\frac{n}{m}},
\end{equation*} where $\xi(m,n)$ is given by $$\xi(m,n)=\prod_{i=m}^{n-1}\Big(1+\frac{1}{2i}\Big) =\frac{\Gamma(n+\frac12)\Gamma(m)}{\Gamma(m+\frac12)\Gamma(n)}.$$ Then $X=(X(n))_{n\geq m}$ and $Y=(Y(n))_{n\geq m}$ are submartingales. If $f$ is affine, then they are martingales.
\end{lem}
\begin{proof}
Fix $n\geq m$ and let $\Delta \Zz[m,n]=\Zz[m,n+1]-\Zz[m,n]$. The martingale property of $X$ for an affine attachment rule $f(x)=\frac12 x+\beta$ follows immediately from 
	\begin{equation}	
	\begin{aligned}\label{eq:derivemart}
		\EE^k \big[f(\Zz[m,n+1]) \big| \Zz[m,n] \big]
		&=\EE^k \big[f(\Zz[m,n])+\sfrac12\1{\{n+1\rightarrow m\}} \big|\Zz[m,n] \big]\\
		&=\EE^k \big[f(\Zz[m,n])|\Zz[m,n] \big]+\sfrac12\EE\big[\sfrac{f(\Zz[m,n])}{n}\big|\Zz[m,n]\big]\\
		&=\big( 1+\sfrac{1}{2n}\big) f(\Zz[m,n]).
	\end{aligned}
	\end{equation}
The corresponding calculation for $Y$ is performed in complete analogy to \eqref{eq:derivemart}, we obtain
	\begin{equation*}
		\EE\big[ f(\Zz[m,n+1])^2 \big|\Zz[m,n] \big]
=\big(1+\sfrac{1}{n}\big) f(\Zz[m,n])^2+\sfrac{1}{4n}f(\Zz[m,n]),
	\end{equation*}
and thus
	\begin{align*}
		\sfrac{n+1}{m}\EE\big[ Y(n+1) \big|\Zz[m,n] \big]
=&\;\big(1+\sfrac{1}{n}\big) f(\Zz[m,n])^2+\sfrac{1}{4n}f(\Zz[m,n])+\sfrac12\big(1+\sfrac{1}{2n}\big) f(\Zz[m,n])\\
		=&\;\big(1+\sfrac{1}{n}\big)\sfrac nm Y(n).
	\end{align*}
Division by $(1+n^{-1})n/m=(n+1)/m$ now yields the martingale property. 
For strictly concave $f$, we have $\Delta f(i)=f(i+1)-f(i)>\frac12$, for all $i\in \NN$, and the equalities in the  above calculations turn into inequalities yielding the submartingale property.
\end{proof}

By \cref{lem:degmart2}, for all $n\geq m\in\NN$ and $k\in\NN$, \begin{equation}\label{eq:xitoroot}\xi(m,n)=\prod_{i=m}^{n-1}\Big(1+\frac{1}{2i}\Big)\in\Big[\sqrt{\frac nm}, (1+\delta(m))\sqrt{\frac nm}\Big],\end{equation} where $\delta(m)$ can be chosen such that $\lim_{m\to\infty}\delta(m)=0.$ In the affine case $\xi(m,n)=\frac1{f(k)}\, {\EE^k f(\Zz[m,n])}$, in particular
the \emph{score} $\xi(m,N)$ of a vertex $m$ is asymptotically proportional to its expected degree at time $N$. For the deviation from the affine case we introduce the notation \begin{equation}\label{eq:PSIDEF}\psi^k(m,n):=\frac{\EE^k f(\Zz[m,n])}{\xi(m,n)}.\end{equation} 
%
%
Determining the magnitude of $\psi^k$ is the first step towards the proof of \cref{thm:PA}.  As we will see later, it suffices to study the special case 
\begin{equation}\label{eq:glimitalpha}
f(k)=\frac{k}{2}+\frac{\alpha}{2}\frac{k}{\log (k \vee \e)} + \beta, \qquad \mbox{ for } k\geq 0,
\end{equation} with $\alpha\geq 0$ and $\beta=f(0)>0.$ 

\begin{prop}[First and second moment upper bound]\label[prop]{prop:expbounds1}
Let $f$ be an attachment rule of the form~\eqref{eq:glimitalpha}. Then, for any $k\in\NN$, there exist constants $C=C(k),C'=C'(k)$ only dependent on $\alpha$ and $\beta$, such that for all pairs $m,n\in\NN$ with $n\geq m$,
$$\EE^k f(\Zz[m,n])\leq C\sqrt{\frac{n}{m}}\big(1 \vee \log\sfrac{n}{m}\big)^{\alpha}$$
and
$$\EE^k f(\Zz[m,n])^2\leq C'\frac{n}{m}\big(1 \vee \log \sfrac{n}{m}\big)^{2\alpha}.$$
\end{prop}

\begin{prop}[First moment lower bound]\label[prop]{prop:expbounds2}
Let $f$ be as in ~\eqref{eq:glimitalpha}. Then,
there exists a constant $c>0$ only dependent on $\alpha$ and $\beta$, such that for all pairs $m,n\in\NN$ with $n\geq m$ and any $k\in\NN\cup\{0\}$
$$\EE^k f(\Zz[m,n])\geq c\, \sqrt{\frac{n}{m}}\big(1 \vee \log\sfrac{n}{m}\big)^\alpha.$$ 
\end{prop}

We note that the two propositions together imply that there are constants $0<c'\leq C'$ depending only on $\alpha$ and~$k$, such that
\begin{equation}\label{psisize}
c' \, \big(1 \vee \log\sfrac{n}m\big)^\alpha \leq \psi^k(m,n) \leq 
C''\, \big(1\vee \log\sfrac{n}m\big)^\alpha.
\end{equation}
To prove ~\cref{prop:expbounds1} and~\cref{prop:expbounds2} we need three auxiliary statements concerning the properties of the attachment rule and the behaviour of the degree evolutions $\mathcal{Z}([m,n])_{n\geq m}$. In particular, in \cite{dereich_random_2009} a scaling function $\Phi$ is introduced to linearise the degree evolutions with respect to logarithmic time. As a byproduct of \cite[Lemma 2.1]{dereich_random_2009}, one obtains useful bounds for the degree evolutions.

\begin{lem}\label[lem]{lem:concave}
Let $f$ be a concave attachment rule and $g$ be given by $$g(x)=\frac{x}{\log({f^{-1}(x))}}, 
\qquad \mbox{ for }x\in \{f(k),k\in \NN\},$$ then there is $K=K(f)\in\NN$, such that $g$ is concave on $\{f(k),k\geq K\}$.

\end{lem}

\begin{proof}

By interpolation we can assume that $f$ is twice differentiable on $(0,\infty)$ with existing right derivative in $0$. Let $e$ denote the inverse of $f$, which is a well defined convex function, since $f$ is increasing and concave. The second derivative of $g$ is given by \begin{equation}\label{eq:gdiff}g''(x)=\frac{x(e'(x))^2(\log e(x)+2)-e(x)\log e(x)(xe''(x)+2e'(x))}{e(x)^2(\log e(x))^3},\end{equation} 
for $x\in[0,\infty)$. To see that $g''(x)\leq 0$ for large $x$, we note that $e''(x)\geq 0$ and
$e'(0)\leq e'(x)(\lim_{k\to\infty}\Delta f(k))^{-1}$. As $e(x)$ is bounded below by $x-1$,  the numerator in \eqref{eq:gdiff} is non-positive for sufficiently large $x$.
\end{proof}

\begin{lem}
\label{lem:DM09Lemma2.1}
Let $f$ satisfy condition \eqref{eq:glimitalpha} with $\alpha>0$ and set $$\Phi(x)=\sum_{i=0}^{x-1}\frac{1}{f(i)},\; x\in\N\cup\{0\}.$$ Then, for fixed $m\in\NN$, the process $$\Big(\Phi(\Zz[m,n])-\sum_{i=m}^{n-1}\frac1i\Big)_{n\geq m}$$ is a martingale.
\end{lem}
\begin{proof}
This is the first statement of \cite[Lemma 2.1]{dereich_random_2009}. Note that in their notation $$t-s=\sum_{i=m}^{n-1}\frac1i,\, \textrm{ and } Z[s,t]=\Phi(\Zz[m,n]).$$
\end{proof}

\begin{lem}\label[lem]{lem:phibounds}
Let $f$ and $\Phi$ be as in \cref{lem:DM09Lemma2.1}. Then,
\begin{enumerate}[(i)]
\item the linear interpolation \smash{$\Phi^{-1}\colon[f(0)^{-1},\infty)\longrightarrow [0,\infty)$} of the inverse of $\Phi$ exists and is strictly monotone, in particular, for $x\geq 1/{f(0)}$ and $k\in\NN,$ $$\Phi^{-1}(x)\geq k,\ \ \textrm{ if } x\geq\Phi(k);$$\label{list_1}
\item there are constants $c,C>0$, only depending on $f$, such that, for all $x\in\NN,$ $$\frac{1}{f(0)}\vee\big(2\log_+ x -2\alpha\log\log_+ x-c\big)\leq\Phi(x)\leq 2\log_+ x -2\alpha\log\log_+x +C,$$ where $\log_+y=\log (y\vee 1)$ and $\log\log_+y=\log\log (y\vee \e), y\in \R.$\label{list_2}
\end{enumerate}
\end{lem}
\begin{proof}
For \textit{(i)} note, that the attachment rule $f$ is positive and strictly increasing, which implies that $\DD \Phi={1}/{f}>0$ is strictly decreasing. Thus $\Phi$ is concave and strictly increasing, hence its inverse is well defined, convex, strictly increasing and $\Phi^{-1}(y)=x$, if $y=\sum_{i=0}^{x-1}{1}/{f(i)}.$ The claimed monotonicity is inherited by the linear interpolation.\linebreak
%
To show \textit{(ii)}, we note that $\Phi(x)\geq 1/{f(0)}$ is true for any $x\in\NN$ and that
$$\frac{1}{f(i)}= \frac{2}{i}-\frac{2\alpha}{i\log i}+{O}\Big(\frac{1}{i(\log i)^2}\Big),$$
from which the statement follows by summation.
\end{proof}

\begin{proof}[Proof of Proposition~\ref{prop:expbounds1}]
We begin with the first moment and note that, for $n\geq n$, $$f(\Zz[m,n+1])=f(\Zz[m,n])+\1 \{n+1\rightarrow m\}\Delta f(\Zz[m,n]))$$ and conditioning on $\Zz[m,n]$ yields $$\EE^k[ f(\Zz[m,n+1])| \Zz[m,n]]=f(\Zz[m,n])+\frac{ f(\Zz[m,n])\Delta f(\Zz[m,n]))}{n}.$$
 Taking expectations we obtain the recursion
\begin{equation}\label{eq:recursion1}
 \EE^k f(\Zz[m,n+1])=\EE^k f(\Zz[m,n])+\EE^k\frac{f(\Zz[m,n])\DD f(\Zz[m,n])}{n}.
\end{equation}
Note that for sufficiently large $i$, $\log i\geq \log f(i)$ and hence \begin{equation}
\label{eq:idep1}
\Delta f(i)=\frac12 +\frac\alpha2\Big(\frac{1}{\log (i+1)} -\frac{1}{(\log i)(\log i +1)}\Big)\leq \frac12 +\frac\alpha2\frac{1}{\log f(i)}.\end{equation}  We may thus fix $i_0$ such that $f(i)>\e^2$ and \eqref{eq:idep1} hold for all $i\geq i_0.$ For $k\geq i_0,$ it follows that
$$\EE^k f(\Zz[m,n])\DD f(\Zz[m,n])\leq \frac{1}{2}\EE^kf(\Zz[m,n])+\frac{\alpha}{2}\EE^k\frac{f(\Zz[m,n])}{\log f(\Zz[m,n])}.$$ The function \smash{$x\mapsto x/{(\log x)}$} is concave on $(\e^2,\infty)$, and we  apply Jensen's inequality to the second term in this sum and obtain $$\EE^kf(\Zz[m,n])\DD f(\Zz[m,n])\leq \frac{1}{2}\EE^kf(\Zz[m,n])+\frac{\alpha}{2}\frac{\EE^kf(\Zz[m,n])}{\log \EE^k f(\Zz[m,n])}.$$ Applying this bound to the right hand side of \eqref{eq:recursion1} yields, after division by $\EE^kf(\Zz[m,n]),$
\begin{equation}
\frac{\EE^kf(\Zz[m,n+1])}{\EE^kf(\Zz[m,n])}\leq 1+ \frac{1}{2n}+\frac{\alpha}{2n\log \EE^kf(\Zz[m,n])}.
\end{equation}
We can apply the lower bound in~\eqref{eq:xitoroot} to bound the denominator of the last term from below by $n\big(1\vee \log(n/m)\big)$ to get
\begin{equation}\label{eq:rec2}
\frac{\EE^kf(\Zz[m,n+1])}{\EE^kf(\Zz[m,n])}\leq 1+ \frac{1}{2n}+\frac{\alpha}{n(1\vee\log\frac{n}{m})}.
\end{equation}
Iterating both sides of \eqref{eq:rec2} in $n$ then yields
$$\EE^k f(\Zz[m,n])\leq f(k)\prod_{i=m}^{n-1}\Big(1+ \frac{1}{2i}+\frac{\alpha}{i(1\vee\log\frac{i}{m})}\Big),$$
and using the inequality $1+x\leq \e^x$ we get $$\EE^k f(\Zz[m,n])\leq f(k)\exp\Big(\sum_{i=m}^{n-1}\frac{1}{2i}+\sum_{i=m}^{n-1}\frac{\alpha}{i(1\vee\log\frac{i}{m})}\Big),$$ which implies
\begin{equation}\label{eq:almbound}
\begin{aligned}
\EE^kf(\Zz[m,n])\leq&\, f(k)\exp\Big[\frac{1}{2}\sum_{i=m}^{n-1}\frac{1}{i}+\alpha\Big(\sum_{i=m}^{\lceil\e m\rceil-1}\frac{1}{i}+\sum_{i=\lceil\e m\rceil}^{n-1}\frac{1}{i\log\sfrac{i}{m}}\Big)\Big].
\end{aligned}
\end{equation}
We have $$\e^{\frac{1}{2}\sum_{i=m}^{n-1}\frac{1}{i}}\leq D\sqrt{\frac{n}{m}},$$ for some constant $D.$ To handle the second expression in the exponent we observe that 
\smash{$\sum_{i=m}^{\lceil\e m\rceil-1}i^{-1}\leq {11}/{6}$} and  
$$\sum_{i=\lceil\e m\rceil}^{n-1}\frac{1}{i\log\sfrac{i}{m}}\leq \int_{\e m}^{n}\frac{\frac{1}{m}}{\frac{s}{m}\log\frac{s}{m}}\di s+C''=\int_{\e}^{\frac{n}{m}}\frac{1}{x\log x}\di x+C''=\log\log \sfrac{n}{m}+C'',$$  
for some absolute constant $C''$. Applying these estimates to \eqref{eq:almbound} we arrive at $$\EE^k\Zz[m,n]\leq f(k)\e^{\frac{11\alpha}{6}+C''}D\sqrt{\frac{n}{m}}\big(1\vee\log\sfrac{n}{m}\big)^{\alpha},$$ proving the desired bound for $C(k)=f(k)\e^{C''+{11\alpha}/{6}}D.$
\medskip

It remains to deduce the bound for the second moment. We argue as for the first moment, conditioning as in the derivation of \eqref{eq:recursion1} yields a similar recursion for the function $f(\cdot)^2$ in terms of $f(\cdot)^2$ itself and the differences $\Delta f(\cdot)^2:=\Delta(f(\cdot)^2)$. In fact we obtain
$$\EE^k f(\Zz[m,n])^2= f(k)^2+\sum_{s=m}^{n-1}\EE^k\frac{f(\Zz[m,s])\Delta f(\Zz[m,s])^2}{s}.$$ Since $f$ is nondecreasing, we find that $\Delta f(k)^2\leq f(k+1)2\Delta f(k)$ and thus $$\EE^k f(\Zz[m,n])^2\leq f(k)^2+\sum_{s=m}^{n-1}\EE^k\frac{2f(\Zz[m,s]+1)^2\Delta f(\Zz[m,s])}{s}=:E(m,n).$$ The function $E(m,n)$ can be bounded in the same fashion as the first moment, we obtain, for $n\geq m$, $$\frac{E(m,n+1)}{E(m,n)}\leq 1+\frac{1}{n}+\frac{2\alpha}{n(1\vee\log\frac{n}{m})},$$ which implies $E(m,n)\leq C'(k)\big(\log(n/m)\big)^{2\alpha} n/{m},$ and the second moment bound follows.\end{proof}

\begin{proof}[Proof of Proposition~\ref{prop:expbounds2}]
By monotonicity, we only need to focus on the lower bound for $k=0$ and begin with the observation that the concavity condition on $f$ implies that
\begin{equation}\label{eq:explb1}
\EE f(\Zz[m,n])\geq f(0)+\sfrac{1}{2}\EE \Zz[m,n]. 
\end{equation}
To obtain a lower bound on $\Zz[m,n]$, we begin by representing $\Phi(\Zz[m,n])=\sum_{i=m}^{n-1}{i^{-1}}+M_n,$ where $(M_n)_{n\geq m}$ is a martingale, using \cref{lem:DM09Lemma2.1}. Clearly, $$\EE \Phi(\Zz[m,n])=\frac{1}{f(0)}+\sum_{i=m}^{n-1}\frac{1}{i},$$ and using concavity of $\Phi,$ Jensen's inequality implies that $$\Phi(\EE\Zz[m,n])\geq\frac{1}{f(0)}+\sum_{i=m}^{n-1}\frac{1}{i},$$ which yields, together with the upper bound on $\Phi$ from \cref{lem:phibounds},
\begin{equation*}
 (C +2\log \EE\Zz[m,n]-2\alpha\log\log \EE\Zz[m,n]) \vee 0 \geq \log\sfrac{n}{m}
\end{equation*}
for some suitably chosen constant $C>0.$ This yields 
\smash{$\EE\Zz[m,n]\geq d\sqrt{{n}/{m}}(\log\EE\Zz[m,n] \vee 0)^\alpha$} for some small constant $d>0$ and combining the last inequality with \eqref{eq:explb1} we obtain $$\EE f(\Zz[m,n])\geq f(0)+\frac{d}{2} \sqrt{\frac{n}{m}}(\log\EE\Zz[m,n] \vee 0)^\alpha. $$ The expectation on the right can be bounded below by the expectation in the affine case, for which a lower bound is implicit in \eqref{eq:xitoroot}. For all sufficiently large $n>m$ we get $$\EE f(\Zz[m,n])\geq f(0)+c'\sqrt{\frac{n}{m}}(1 \vee \log\sfrac{n}{m})^\alpha$$ for some $c'>0$ and a further adjustment of the constant, which only depends on the value $f(0)$, yields the statement of the proposition.\end{proof}

We close this section with two very intuitive stochastic domination results from \cite{dereich_random_2013} which are instrumental in the proof of \cref{thm:PA}.
\begin{lem}[Stochastic domination I, {\cite[Lemma~2.9]{dereich_random_2013}}]\label{lem:DM11Lemma2.9}
Let $f$ be concave and fix integers $m<n_1<\dots<n_i.$ The process $(\Zz[m,n])_{n\geq m}$ conditioned on the event $\{\Delta\Zz[m,n_j]=0, j=1,\dots,i\}$ is stochastically dominated by the unconditioned process.
\end{lem}
\begin{proof} See {\cite[p. 18]{dereich_random_2013}}.
\end{proof}
\begin{lem}[Stochastic domination II, {\cite[Lemma~2.10]{dereich_random_2013}}]\label{lem:DM11Lemma2.10}
Let $f$ be concave and fix $i,k\in\NN$. For integers $n_i>\dots >n_1>m>k+i$ there is a coupling of the process $(\Zz[m,l])_{l\geq m}$ started in $\Zz[m,m]=k$ and conditioned on $\{\DD\Zz[m,n_j]=1 \;\forall j\in\{1,\dots,i\}\}$ and the unconditional process $(\Zz[m,l])_{l\geq m}$ started in $\Zz[m,m]=k+i$ such that for the coupled versions $(\bar{\Zz}^\ssup{c}[m,l],\bar{\Zz}^\ssup{u}[m,l])_{l\geq m}$ one has $$\DD\bar{\Zz}^\ssup{c}[m,l]\leq\DD\bar{\Zz}^\ssup{u}[m,l]+\sum_{j=1}^i\1{\{l=n_j\}}, \textrm{ for all }l\geq m,$$ and consequently $$\bar{\Zz}^\ssup{c}[m,l]\leq \bar{\Zz}^\ssup{u}[m,l], \textrm{ for all }l\geq m.$$ I.e. the unconditioned process initiated in $k+i$ dominates the process initiated at $k$ and conditioned to have jumps at times $n_1,\dots,n_i.$
\end{lem}
\begin{proof}
The case $i=1$ is the original statement {\cite[Lemma~2.10]{dereich_random_2013}} and proven there. The generalisation to $i=2,3,\dots$ is obtained by a straightforward induction argument.
\end{proof}
\subsection{Lower bounds for distances}
The first moment estimates of the previous section now yield lower bounds on the typical distances in a straightforward manner under the assumption of bounded correlation for edges along any self-avoiding path.
\begin{lem}[First order lower bound on distances]\label[lem]{lem:1stMomentBound}
Let $\Gg_N$ be a random graph with vertex set $[N]$ and assume that there are $\kappa_N\geq0$ and $\Psi_N\geq 0$, such that, for any self-avoiding path $P=(v_0,\dots,v_l)$, we have
\begin{equation}
\label{eq:cor}
\PP(P\subset \Gg_N)\leq \kappa_N^l\prod_{j=0}^{l-1}\PP(v_j\leftrightarrow v_{j+1})
\end{equation}
and
\begin{equation}
\label{eq:conpb}\PP(v\leftrightarrow w)\leq\frac{\Psi_N}{\sqrt{vw}},\ \ \text{for all }v,w\in[N],
\end{equation}
where
\begin{equation}\label{eq:Psiub}
\liminf_{N\to\infty}\kappa_N \Psi_N\log N>1.
\end{equation}
Then, for uniformly chosen vertices $U,V\in\Gg_N$, $$\lim_{N\to\infty}\PP\Big(d_N(U,V)\geq\Big\lceil\frac{\log N}{\log\log N+\log \Psi_N+\log\kappa_N}\Big\rceil\Big)=1.$$
\end{lem}

\begin{proof}
We first observe that for any positive sequence $(a_i)_{i=0}^\infty$ satisfying ${a_{i+1}}/{a_i}\geq 1+\delta,$ for all $ i\geq 0$ and some fixed $\delta>0$, we can find a constant $C>0$ with \begin{equation}\label{eq:auxState}\sum_{i=0}^K a_i\leq C a_K, \ \ \textrm{for all } K\in\NN.\end{equation} 
Let $1\leq l\leq L=L(N)=\lfloor {\log N}/{(\log\log N+\log(\kappa_N\Psi_N))}\rfloor$ and $P=(v_0,\dots,v_l)$ be self-avoiding. Assumptions \eqref{eq:cor} and \eqref{eq:conpb} imply that
\begin{align*}
&\PP(P\subset \Gg_N)\leq \kappa_N^l\prod_{j=0}^{l-1}\frac{\Psi_N}{\sqrt{v_jv_{j+1}}}\leq\frac{\big(\kappa_N\Psi_N\big)^l}{\sqrt{v_0v_l}}\prod_{j=1}^{l-1}\frac{1}{v_j}.
\end{align*}
For $v,w\in[N]$ and $\mathcal{P}_l(v,w)$ denoting the set of self-avoiding paths of length $l$ from $v$ to $w$,
\begin{equation*}
\begin{aligned}
 \PP(d_N(v,w)\leq L)\leq & \sum_{l=1}^{L}\sum_{(v_0,\dots,v_l)\in \mathcal{P}_l(v,w)}\frac{\big(\kappa_N\Psi_N\big)^l}{\sqrt{vw}}\prod_{j=1}^{l-1}\frac{1}{v_j} \\ 
\leq & \sum_{l=1}^{L}\frac{\big(\kappa_N\Psi_N\big)^l}{\sqrt{vw}}\Big(\sum_{j=1}^N\frac{1}{j}\Big)^{l-1}\leq\frac{1}{\sqrt{vw}}\sum_{l=1}^{L}\kappa_N^l\Psi_N^{l}(\log N)^{l-1}.
\end{aligned}
\end{equation*}
By \eqref{eq:Psiub}, the terms in the last sum grow at least exponentially in $l$ for all sufficiently large $N$, so using \eqref{eq:auxState} we infer the existence of an independent constant $C>0$ such that
\begin{equation}\label{eq:pathbd}
 \PP(d_N(v,w)\leq L)\leq C\frac{\big(\kappa_N\Psi_N\log N\big)^L}{\sqrt{vw}\log N}.
\end{equation}
For any $\eE\in(0,1)$, the probability that one of the vertices $U,V$ is smaller than ${\eE}/{3} N$ is bounded by ${2\eE}/{3}$ and thus using \eqref{eq:pathbd} on the complement of this event results in
\begin{align*}
\PP(d_N(V,W)\leq L) & \leq \sum_{v,w\geq \frac{\eE}{3} N}\PP(d_N(v,w)\leq L)\PP(V=v,W=w)+\frac{2\eE}{3}\\
&\leq 3C\frac{\big(\kappa_N\Psi_N\log N\big)^L}{\eE N\log N}+\frac{2\eE}{3} = \frac{3C}{\eE \log N}\e^{L[\log\log N +\log(\kappa_N\Psi_N)]-\log N} +\frac{2\eE}{3}\\ 
& \leq \frac{3C}{\eE\log N}+\frac{2\eE}{3},
\end{align*} 
and the proof is complete.
\end{proof}

The lower bounds on the distances in \cref{thm:PA} can now be obtained by verifying the assumptions of \cref{lem:1stMomentBound}.
\begin{prop}[Lower bounds for PA]\label[prop]{prop:lowerbounds}
The preferential attachment model $\Gg_N$ with attachment rule $f$ of the form \eqref{eq:fcondition} satisfies $$\lim_{N\to\infty}\PP\Big(d_N(U,V)\geq \Big(\frac{1}{1+\alpha}-\delta\Big)\frac{\log N}{\log\log N}\Big)=1,$$ for every $\delta>0$ and independently and uniformly chosen vertices $U,V\in\Gg_N.$
\end{prop}
\begin{proof}
Let $P=(v_0,\dots,v_n)$ be a self-avoiding path along vertices in $[N].$ By definition of the preferential attachment mechanism $\PP(P\subset\Gg_N)$ can be decomposed in the following way: each edge $(u,v)$ in $P$ corresponds to a jump in the degree evolution of the vertex $u\wedge v$ and since $P$ is self-avoiding, any given degree evolution can feature at most twice in the formation of $P$. Moreover, if a degree evolution is used twice, then it is used to obtain two consecutive edges of $P$. By independence of the degree evolutions, $\PP(P\subset \Gg_N)$ therefore must factorise into terms of the form $\PP(u\rightarrow v \leftarrow w)$ and $\PP(u\rightarrow v)$, corresponding to two jumps and one jump of the repsective degree evolution. To obtain a bound on $\PP(u\rightarrow v \leftarrow w)$ fix $v<u<w$. By \cref{lem:DM11Lemma2.10}, the process $(\Zz'[v,n])_{n\geq v}$, started at $\Zz'[v,v]=1$ and evolving according to the law of an unconditioned degree evolution, stochastically dominates the process $(\Zz[v,n])_{n\geq v}$ conditional on $\Zz[v,u]=1$ and hence $$\PP(\Delta\Zz[v,w]=1|\Delta \Zz[v,u]=1)\leq \PP^1(\Delta\Zz[v,w]=1).$$ We obtain $$\PP(u\rightarrow v \leftarrow w)=\PP(\Delta\Zz[v,w]=\Delta \Zz[v,u]=1)\leq \PP(\Delta\Zz[v,u]=1)\PP^1(\Delta\Zz[v,w]=1),$$ and in combination with \cref{prop:expbounds1} this shows that the edge correlation bound \eqref{eq:cor} is satisfied with $\kappa_N={C_{\ref{prop:expbounds1}}(1)}/{C_{\ref{prop:expbounds1}}(0)}.$ According to Proposition~\ref{prop:expbounds1}, we also have $$\PP(w\rightarrow v)=\frac{\EE f(\Zz[v,w])}{w}\leq \frac{C_{\ref{prop:expbounds1}}(0)(\log\sfrac wv )^\alpha}{\sqrt{vw}} $$ and thus the bound \eqref{eq:conpb} holds for $\Psi_N=C_{\ref{prop:expbounds1}}(0)(\log N)^\alpha$, in the case where the attachment rule $f$ is of the form \eqref{eq:glimitalpha}. For such $f$ the distance bound follows therefore for any choice of $\delta\in(0,\frac{1}{1+\alpha})$ immediately from \cref{lem:1stMomentBound}.
\medskip\pagebreak[3]

For $f$ of the more general form \eqref{eq:fcondition}, we note that $\bar{f}\geq f$ implies that the respective networks satisfy $\bar{\Gg}_N\geq \Gg_N$ stochastically for all $N\in\NN$, where $\geq$ is the partial order given by inclusion on the edge sets of graphs with the same vertex set, so that distances in $\Gg_N$ dominate those in $\bar{\Gg}_N$. By~\eqref{eq:fcondition}, for every $\varepsilon>0$, there is $k_0\in\NN$ such that, for all $k\in\NN_0$, $$f(k)\leq f(k_0)+\frac{k}{2}+\frac{\alpha+\varepsilon}{2}\frac{k}{1 \vee \log k}=:\bar{f}(k).$$ Choosing $\varepsilon$ suitably in dependence on $\delta$ thus allows to deduce the bound for general $f$ from the special case treated in the previous paragraph.\end{proof}

\section{Proof of upper bounds -- preferential attachment}
To prove the upper bound of Theorem~\ref{thm:PA} we need to find short paths connecting two uniformly chosen vertices, say $U$ and $W$. We use the concept of an inner core: we will show that with high probability $U$ and $W$ have at most  distance $ (1+o(1))(2+2\alpha)^{-1} \log N/\log \log N$ to a small set of vertices that has uniformly bounded diameter, see for instance \cite{chung2003,dommers_diameters_2010} for similar ideas.

Starting from a uniform vertex $U\in\Gg_N$ we perform essentially a breadth-first search, a precise definition of the exploration algorithm is given below. Roughly speaking, at each exploration stage $k$ the set of vertices at distance $k$ from $U$ is assessed using the the score $\xi$ introduced in \cref{lem:degmart2}, i.e. for a set $V\subset[N]$ and $p\in\N$ we call $\xi^p(V,N):=\sum_{v\in V}\xi(v,N)^p$ the total $p$-score of the set $V$. The proof is based on the following three auxiliary results.
\begin{itemize}
\item By a local approximation argument we first show that with high probability either the local exploration around $U$ will quickly lead to a configuration with a high score or the vertex is in a small component, see \cref{prop:localknown}.
\item Using moment estimates we show that starting in a configuration with sufficiently high score, the score will quickly grow from generation to generation with high probability, see \cref{prop:scoregrowth}, until we find a configuration with score exceeding $\sqrt{N}/(\log N)^{2\alpha+2}$.
\item Finally, we show that a subset with score exceeding $\sqrt{N}/N^{2\alpha+2}$ is with high probability connected to a dense subgraph among the oldest vertices. This subgraph is of bounded diameter, see \cref{prop:cores}.
\end{itemize}

Recall our notations $\xi(\cdot, \cdot)$ and $\psi^k(\cdot,\cdot)$, introduced in \cref{lem:degmart2} and \eqref{eq:PSIDEF}, respectively, which are repeatedly used throughout the following sections. If the graph size~$N$ is fixed, we also write $\xi(\cdot),\psi^k(\cdot)$ for $\xi(\cdot,N),\psi^k(\cdot,N)$ for ease of notation. Note that, for $m\leq n\leq N$, \eqref{eq:xitoroot} allows us to appproximate the ratios $\xi(m,n)/n$ as $$\frac{\xi(m,n)}{n}\approx \frac{\xi(m)\xi(n)}{N},$$ we use this approximate factorisation frequently in subsequent proofs. Here and throughout the article, `$f_1(\cdot)\approx f_2(\cdot)$' means that the ratio of the functions $f_1,f_2$ is bounded away from $0$ and $\infty$ uniformly in all arguments.  

\subsection{Local approximation results -- initial phase}

A \emph{configuration} $ \mathbf{e}$ associates with every vertex  a state in the set
$\{$veiled, active, dead$\}$, and with every potential edge a state in
the set $\{0,1,$ unknown$\}$, the state `unknown' capturing the absence of the information whether an edge is contained in $\Gg_N$ or not. The graph associated with a configuration consists of the
vertex set $[N]$ and all edges in state~$1$. The score of a configuration
is the cummulative score of all active vertices in the configuraton.
\medskip

We now describe the \emph{exploration process} that we follow in the initial phase as well
as the main phase. Its definition uses a non-increasing sequence $(\ell_k)_{k\in\N}$ of truncation levels,
which are set to $\ell_k=1$, for all $k\in\N$, in the initial phase.
The exploration is an inhomogeneous Markov chain $(\mathcal{E}_k)_{k \in \N}$ on the space of configurations, which we define on the probability space associated with the random graph $\Gg_N$.
We assume that we start with an initial configuration~$\mathcal{E}_0$, and the graph  associated with 
this configuration is a tree.%
\medskip%

In the $k$th exploration step we go through all active vertices in $\mathcal{E}_{k-1}$, starting with the vertex 
of smallest label and proceeding in increasing order of labels until all active vertices are treated. 
For each such vertex $v$ we
\begin{enumerate}
\item inspect all potential edges connecting~$v$ to veiled vertices in $\{\ell_k,\ldots,N\}$;
\item If the edge does not exist in $\Gg_N$ its state becomes $0$ and the veiled vertex remains so;
\item If it does exist in $\Gg_N$ its state becomes $1$ and the veiled vertex is declared \emph{pre-active}.
\end{enumerate}
Once all active vertices are explored, they are declared dead, the pre-active vertices are declared active and the  exploration step ends. Note that, if we start with a configuration associated with a tree, the configuration at the end of an exploration step is again associated with a tree. We call such configurations \emph{proper}.
The sets of active, veiled and dead vertices of $ \mathbf{e}$ are denoted by $\mathrm{active}( \mathbf{e}), \mathrm{veiled}( \mathbf{e})$ and $\mathrm{dead}( \mathbf{e}),$ respectively.
\medskip

The following proposition (and nothing else in this paper) relies on a coupling of local neighbourhoods in~$\mathcal{G}_N$ with the `idealised neighbourhood tree'  introduced in \cite[Section 1.3]{dereich_random_2013}. The probability that this tree is infinite is denoted by $p(f)$. It coincides with  the asymptotic proportion of vertices in the connected component of a uniformly chosen vertex, and hence with the probability that such a vertex is in the giant component.  \pagebreak[3]

\begin{prop}\label[prop]{prop:localknown}
Suppose $U\in\Gg_N$ is uniformly chosen, determining an initial configuration in which 
$U$ is active, all other vertices are veiled and all edges are in state unknown. Denote by 
$\xi(V):=\sum_{v\in V} \xi(v)$ the score associated with a set $V\subset[N]$ of vertices. 
Given $\eE>0$ and $s_0>0$ there exists $k_0=k_0(s_0,\eE)\in\NN$, such that, for sufficiently large $N,$
we have
\begin{equation*}
\begin{aligned}[alignment]
\PP\big(\textrm{there exists some } k\leq k_0 \textrm{ and a set } & A\subset\mathrm{active}(\Ee_k) \textrm{ satisfying } \xi(A)\geq s_0  \,\xi(\min A) \big)\\
&\geq\; p(f)-\eE.
\end{aligned}
\end{equation*} 
\end{prop}

As the proof of \cref{prop:localknown} is obtained by application of the results of \cite{dereich_random_2013} and is therefore not self-contained we defer it to \cref{knownresults}.

\subsection{Score growth -- main phase}\label{sec:truncsec}

Our next goal is to fix a sequence $(\ell_k)_{k\geq 1}$ which guarantees that the score of encountered configurations during an exploration of the giant component grows with high probability at a certain deterministic rate. We rely on a careful analysis of the exploration process and the following concentration inequality.
\begin{lem}[Lower tail bound for independent sums, {\cite[Theorem 2.7]{chung_complex_2006}}]\label[lem]{lem:CL27} Let $I$ be a finite set and $(X_i)_{\in I}$ be independent, nonnegative random variables. Then, for any $\lambda>0$, $$\PP\Big(\sum_{i\in I}X_i\leq \sum_{i\in I}\EE X_i-\lambda \Big)\leq \e^{-\frac{\lambda^2}{2\sum_{i\in I}\EE X^2_i}}.$$ 
\end{lem}

We start the main phase in a proper configuration $\mathcal{E}_0$ with the property that  the score of the set $A$ of active vertices in the configuration satisfies $\xi(A)\geq s_0\xi(\min A)$ from some $s_0$ to be specified later. From this initial configuration
we restart the exploration process $( \mathcal{E}_k \colon k\in\N)$  using a new truncation sequence $(\ell_k)_{k\in\N}$.  As before, each $ \mathcal{E}_k$ is a proper configuration. 
While obtaining gradually more information about $\Gg_N$, we need to control the correlation between discovered edges. This is done in the following two lemmas, which provide upper and lower bounds on conditional jump probabilities of a degree evolution $\Zz[m,\cdot]$ given disjoint sets $I_1,I_0$ of times at which $\Zz[m,\cdot]$ is known to jump or to stay constant, respectively.

\begin{lem}[Lower bound for conditional jump probabilities]\label[lem]{lem:lbCond}
For every $k\in\N$ there exists $n_0\in\N$ and a constant $C(k)>0$ 
such that for every $n_0 \leq m \leq N$,
and disjoint sets $I_0,I_1\subset\{m,\dots,N-1\}$ with $\#I_1\leq k-1$ and
\begin{equation}\label{def:thin}\xi(m)\xi(I_0)\leq C(k)\, \frac{N}{2\psi^{k}(n_0,N)},\end{equation} 
the events $A_i:=\{\Delta\Zz[m,l]=\1\{i=1\}$  for all $l\in I_i\}$,  for $i\in\{0,1\}$, satisfy
$$\PP(\DD\Zz[m,j]=1|A_0,A_1)\geq\sfrac 12\PP(\DD\Zz[m,j]=1|A_1)\ \ \text{ for all }j \in\{m,\dots,N-1\} \setminus I_0.$$
\end{lem}

\begin{proof}
Let $j \in\{m,\dots,N-1\} \setminus I_0.$ We have
\begin{equation}\label{eq:lemlbcond1}
\begin{aligned}
\PP(\DD\Zz[m,j]=1,& \DD\Zz[m,l]=0\;\forall l\in I_0|A_1) \\ & =\,\PP(\DD\Zz[m,j]=1|A_1)
-\PP(\DD\Zz[m,j]=1,\;\exists l\in I_0: \DD\Zz[m,l]=1|A_1)\\
& \geq \,\PP(\DD\Zz[m,j]=1|A_1)-\sum_{l\in I_0}\PP(\DD\Zz[n,j]=\DD\Zz[m,l]=1|A_1).
\end{aligned}
\end{equation}
The last sum can be rewritten
\begin{equation}\label{eq:lemlbcond2}
\begin{aligned}
\sum_{l\in I_0}\PP(\DD\Zz[m,j]=&\DD\Zz[m,l]=1|A_1)\\
&=\,\PP(\DD\Zz[m,j]=1|A_1)\sum_{l\in I_0}\PP(\DD\Zz[m,l]=1|A_1,\DD\Zz[m,j]=1 ).
\end{aligned}
\end{equation}
The conditioning event in the last sum involves at most $k$ jumps. We may apply \cref{lem:DM11Lemma2.10} to move them to the start of $\Zz[m,\cdot]$ and then the estimates \eqref{psisize} and \eqref{eq:xitoroot} to
obtain a constant $C(k)$ such that, for all $l\in I_0,$
\begin{equation}
\label{eq:lemlbcond3}
\begin{aligned}
\PP(\DD\Zz[m,l]= 1|A_1,\DD\Zz[m,j]=1) & \leq\,\PP^{k}(\DD\Zz[m,l]=1 )
=\,\frac{\psi^{k}(m,l)\xi(m,l)}{l}\\
&\leq\, C(k)\, \frac{\psi^{k}(n_0,N)\xi(m)\xi(l)}{N}.
\end{aligned}
\end{equation}
Inserting \eqref{eq:lemlbcond3} into~\eqref{eq:lemlbcond2} in combination with~\eqref{eq:lemlbcond1} yields
\begin{align*}\PP(\DD\Zz[m,j]=1&,\DD\Zz[m,l]=0\;\forall l\in I_0|A_1)\\ &\geq \PP(\DD\Zz[m,j]=1|A_1)\Big(1- C(k)\, \frac{\psi^{k}(n_0,N)\xi(m)\xi(I_0)}{N}\Big),\end{align*}
and using \eqref{def:thin} yields the statement.
\end{proof}

\begin{lem}[{Upper bound for conditional jump probabilities}]\label[lem]{lem:Mod212}
For every $k\in\N$ there exists $n_0\in\N$ and $C>0,$ such that for $n_0\leq m \leq N$
and   $I_0, I_1\subset\{m,\dots,N\}$ disjoint satisfying \eqref{def:thin} and  $\#I_1\leq k$, 
$$\PP(\DD\Zz[m,j]=1|A_1,A_0)\leq C\,
\PP(\DD\Zz[m,j]=1|A_1),\ \ \text{for all $j\in\{m,\dots,N-1\}$}.$$
\end{lem}
\begin{proof}
This is a modification of {\cite[Lemma 2.12]{dereich_random_2013}}. We have $$\PP(\DD\Zz[m,j]=1|A_0,A_1)\leq \frac{\PP(\DD\Zz[m,j]=1|A_1)}{\PP(A_0|A_1)},$$ so it suffices to bound 
 $\PP(A_0|A_1)$ uniformly from below. Since $\#I_1\leq k$, we get by \cref{lem:DM11Lemma2.10},
\begin{equation}\label{eq:Mod212lb}
\PP(\DD\Zz[m,j]=0\;\forall j\in I_0|\DD\Zz[m,j]=1\;\forall j\in I_1)\geq\PP^k(\DD\Zz[m,j]=0\;\forall j\in I_0).
\end{equation}
Denoting $i=\min{I_0}$, we obtain
\begin{align*}
\PP^k(\DD\Zz[m,j]=0\;\forall j\in I_0)=&\,\PP^k(\DD\Zz[m,j]=0\;\forall j\in I_0\setminus\{i\}|\DD\Zz[m,i]=0)\\
&\,\times\PP^k(\DD\Zz[m,i]=0)\\
\geq&\,\PP^k(\DD\Zz[m,j]=0\;\forall j\in I_0\setminus\{i\})\PP^k(\DD\Zz[m,i]=0),
\end{align*}
using Lemma~\ref{lem:DM11Lemma2.9}. Iteration yields
\begin{equation}\label{eq:Mod212lb2}
\begin{aligned}
\PP^k(\DD\Zz[m,j]=0\;\forall j\in I_0)&\geq\prod_{j\in I_0}\PP^k(\DD\Zz[m,j]=0)
=\prod_{j\in I_0}\big(1-\sfrac1j {\EE^k f(\Zz[m,j])}\big),
\end{aligned}
\end{equation}
and inserting \eqref{eq:Mod212lb2} into \eqref{eq:Mod212lb} yields
\begin{equation}
\label{eq:Mod212lb3}
\PP(\DD\Zz[m,j]=0\;\forall j\in I_0|\DD\Zz[m,j]=1\;\forall j\in I_1)\geq\prod_{j\in I_0}\big(1-\sfrac1j {\EE^k f(\Zz[m,j])}\big).
\end{equation}
Choose  $n_0$ large enough such that $n_0\leq m\leq j$ implies ${j}^{-1}\EE^k f(\Zz[m,j])<1$.
It is now possible to find~$c>1$ such that 
\smash{$-\log(1-j^{-1} {\EE^k f(\Zz[m,j])})\leq {c}j^{-1} {\EE^k f(\Zz[m,j])}$. }
Thus, taking the logarithm in \eqref{eq:Mod212lb3}, we bound, using \eqref{eq:xitoroot}, 
for some constant $C>0$,
\begin{align*}
-\log\PP(A_0|A_1)\leq & \, \sum_{j\in I_0} \sfrac{c}{j} \EE^k f(\Zz[m,j])=c \sum_{j\in I_0}\frac{\psi^k(m,j)\xi(m,j)}{j}\leq C\frac{\xi(m)\psi^k(m,N)\xi(I_0)}{N}, 
\end{align*}
and the last expression is uniformly bounded by \eqref{def:thin}.
\end{proof}

To choose $(\ell_k)_{k\geq 1}$ suitably, we need to understand how the choice of cutoff points influences the growth of the score. To this end let $ \mathcal{E}$ denote a configuration obtained after some stage of the exploration process, $V\subset\mathrm{veiled}( \mathcal{E})$ and consider the random variable $$S(V)=\xi(\{v\in V: v\leftrightarrow \mathrm{active}( \mathcal{E})\})=\xi(\{v\in V: \exists a\in \mathrm{active}( \mathcal{E}): v \leftrightarrow a\}).$$ The inclusion-exclusion principle yields the lower bound
\begin{equation}
\label{eq:Sincexcl}
S(V)\geq \sum_{v\in V}\xi(v)\sum_{a\in\mathrm{active}( \mathcal{E})}\1\{a\leftrightarrow v\} - \sum_{v\in V}\xi(v)\sum_{\heap{a<b}{a,b\in\mathrm{active}( \mathcal{E})}}\1\{a\leftrightarrow v \leftrightarrow b\}.
\end{equation}
To derive bounds on the probability that the term after the minus sign is positive, we define the events
\begin{align*}
A_1(V, \mathcal{E})&:=\{\exists v\in V; a<b; a,b\in\mathrm{active}( \mathcal{E}):\Delta\Zz[v,a-1]=\Delta\Zz[v,b-1]=1\},\\
A_2(V, \mathcal{E})&:=\{\exists v\in V; a<b; a,b\in\mathrm{active}( \mathcal{E}):\Delta\Zz[a,v-1]=\Delta\Zz[b,v-1]=1\},\\
A_3(V, \mathcal{E})&:=\{\exists v\in V; a<b; a,b\in\mathrm{active}( \mathcal{E}):\Delta\Zz[a,v-1]=\Delta\Zz[v,b-1]=1\}.\\
\end{align*}

Recalling that $\xi^2(A)=\sum_{v\in A}\xi(v)^2$ for $A\subset[N]$, we obtain the following bounds:
\begin{prop}[Collision probability]\label[prop]{prop:colbd}
Let $ \mathbf{e}$ be a proper configuration and $V\subset\mathrm{veiled}( \mathbf{e})$ such that, for some fixed $k\in\N$ and $n_0=n_0(k)$ as in \cref{lem:Mod212}, \begin{equation}\label{eq:colbd} \xi(\min V)\xi(\mathrm{active}( \mathbf{e})\cup \mathrm{dead}( \mathbf{e}))\leq C(k)\,  \frac{N}{2\psi^k(n_0,N)}.\end{equation} Then there is a constant $C>0$, depending only on $f$ and $k$, such that $$\PP\Big(\bigcup_{i=1}^3 A_i(V, \mathcal{E})\Big| \mathcal{E}= \mathbf{e}\Big)\leq C \big( 1 \vee \log\sfrac{N}{\min V \wedge \min (\mathrm{active}( \mathbf{e}))} \big)^{2\alpha+1}\frac{\xi(\mathrm{active}( \mathbf{e}))^2-\xi^2(\mathrm{active}( \mathbf{e}))}{N},$$
\end{prop}

\begin{proof}  Repeated use of the union bound yields $$\PP(A_1(V, \mathcal{E})| \mathcal{E}= \mathbf{e})\leq \sum_{v\in V}\sum_{\heap{a,b\in\mathrm{active}( \mathbf{e})}{a<b}}\PP(\Delta\Zz[v,a-1]=\Zz[v,b-1]=1| \mathcal{E}= \mathbf{e})$$
To drop the conditioning, we first use \cref{lem:Mod212} to remove all dependencies on non-existing connections given by $\mathbf{e}$ and then
\cref{lem:DM11Lemma2.10} to move the jump of $\Zz[v,\cdot]$  to the start. Note that we are allowed to do this as condition \eqref{eq:colbd} and the monotinicity of $\xi$ ensure that \eqref{def:thin} is satisfied, since certainly $\mathrm{active}( \mathbf{e})\cup \mathrm{dead}( \mathbf{e})$ contains the set of continuity points $I_0$ appearing in the conditioning of $\Zz[v,\cdot]$.
\begin{align*}
\PP( &\, \Delta\Zz[v,a-1]= \Delta\Zz[v,b-1]=1| \mathcal{E}= \mathbf{e})\\
&=\PP(\Delta\Zz[v,b-1]=1| \mathcal{E}= \mathbf{e},\Delta\Zz[v,a-1]=1)\, \PP(\Delta\Zz[v,a-1]=1| \mathcal{E}= \mathbf{e})\\
&\leq C_{\ref{lem:Mod212}}^2\PP^1(\Delta\Zz[v,b-1]=1)\PP(\Delta\Zz[v,a-1]=1)
\end{align*}
Using \cref{prop:expbounds1} yields a constant $C$, such that
\begin{align*}
\PP( &\, \Delta\Zz[v,a-1]\\
&\leq C_{\ref{lem:Mod212}}^2C \frac{(\log\frac{b}{v}\vee 1)^\alpha(\log\frac{a}{v}\vee 1)^\alpha}{v\sqrt{ab}}\\
&\leq B \frac{\xi(a)\xi(b)}{v N} (\log\sfrac{b}{v}\vee 1)^\alpha(\log\sfrac{a}{v}\vee 1)^\alpha,
\end{align*}
where the last inequality follows by using \eqref{eq:xitoroot} and combining all occurring constants into $B>0$. Hence, with $v_0=\min V$, we get
$$\PP(A_1(V, \mathcal{E})| \mathcal{E}= \mathbf{e})\leq \frac{B}{N}(\log\sfrac{N}{v_0}\vee 1)^{2\alpha}\sum_{v\in V}\frac 1v \sum_{\heap{a,b\in\mathrm{active}( \mathbf{e})}{a<b}}\xi(a)\xi(b).$$
For $A_2(V, \mathcal{E})$ we need to take into account that, for $a\in \mathrm{active}(\mathbf{e})$, 
$\Zz[a,\cdot]$ may only be conditioned to have at most one jump. This holds since $ \mathbf{e}$ is proper, i.e. the active and dead vertices of $ \mathbf{e}$ together with the explored edges form a tree implying that exactly one edge incident to $a$ has been explored. Using this fact to derive an upper bound on the number of jumps appearing in the conditioning of $\Zz[a,\cdot]$, a similar calculation as above yields $$\PP(A_2(V, \mathcal{E})| \mathcal{E}= \mathbf{e})\leq \frac{B'}{N}(\log\sfrac{N}{a_0}\vee 1)^{2\alpha}\sum_{v\in V}\frac 1v \sum_{\heap{a,b\in\mathrm{active}( \mathbf{e})}{a<b}}\xi(a)\xi(b),$$ for some $B'>0$ and $a_0=\min({\mathrm{active}( \mathbf{e})}).$
Analogously, we obtain
$$\PP(A_3(V, \mathcal{E})| \mathcal{E}= \mathbf{e})\leq \frac{B''}{N}(\log\sfrac{N}{a_0\wedge v_0}\vee 1)^{2\alpha}\sum_{v\in V}\frac 1v \sum_{\heap{a,b\in\mathrm{active}( \mathbf{e})}{a<b}}\xi(a)\xi(b),$$ for some $B''>0.$
Setting $B'''=\max(B,B',B'')$ these three estimates together with the union bound yield
$$\PP\Big(\bigcup_{i=1}^3 A_i(V, \mathcal{E})\Big| \mathcal{E}= \mathbf{e}\Big)\leq \frac{B'''}{N}(\log\sfrac{N}{a_0\wedge v_0}\vee 1)^{2\alpha}\big(\xi(\mathrm{active}( \mathbf{e}))^2-\xi^2(\mathrm{active}( \mathbf{e}))\big)\sum_{v\in V}\frac 1v,$$
which implies the claimed upper bound.
\end{proof}
\begin{rmk}
Note that we only use the that $\mathbf{e}$ is proper to make sure that an active vertex has at most one explored adjacent edge. Our proofs still work, if we drop the requirement that the explored subgraph is a tree and replace it with the requirement that its indegree is bounded in $N$.
\end{rmk}
\cref{prop:colbd} allows us to ignore the second sum of \eqref{eq:Sincexcl} outside a set of small probability. Decomposing the first sum of \eqref{eq:Sincexcl} according to the orientation of the occuring edges yields 
\begin{align*}\sum_{v\in V}\xi(v)\sum_{a\in\mathrm{active}( \mathcal{E})}\1\{a\leftrightarrow v\} &= \sum_{v\in V}\sum_{a\in\mathrm{active}( \mathcal{E})}\xi(v)\1\{v\leftarrow a\}+\sum_{a\in\mathrm{active}( \mathcal{E})}\sum_{v\in V}\xi(v)\1\{a\leftarrow v\}\\
&=:S^<(V)+S^>(V).\end{align*} Setting $$X_v:=\sum_{a\in\mathrm{active}( \mathcal{E})}\xi(v)\1\{v\leftarrow a\},\, v\in V \textrm{ and } Y_a:=\sum_{v\in V}\xi(v)\1\{a\leftarrow v\},\, a\in \mathrm{active}( \mathcal{E}),$$ we note that due to the independence of indegree evolutions $S^<(V)= \sum_{v\in V}X_v$ and $S^>(V)=\sum_{a\in\mathrm{active}( \mathcal{E})}Y_a$ are independent and both are sums of elements of the collection $\{X_v,Y_a:\,v\in V, a\in  a\in\mathrm{active}( \mathcal{E})\}$ of mutually independent random variables. In order to apply \cref{lem:CL27} we determine moment bounds for $X_v,v\in V$, $Y_a,a\in \mathrm{active}( \mathcal{E}).$
\begin{prop}[First and second moments of vertex scores]
Let $ \mathbf{e}$ be a proper configuration and $V\subset\mathrm{veiled}( \mathbf{e})$ such that \eqref{eq:colbd} is satisfied for some $k\in\NN$.
\begin{enumerate}[(i)]
\item There are constants $0<c,C<\infty$ depending only on $f$ and $k$, such that for all $v\in V$
\begin{equation}\label{eq:X1mom}
\EE[X_v| \mathcal{E}= \mathbf{e}]\;\geq\; c\frac{\xi(v)^2}{N}\sum_{\heap{a\in \mathrm{active}( \mathbf{e}):}{a>v}}\xi(a)(\log\sfrac av\vee 1)^\alpha
\end{equation}
and 
\begin{equation}
\begin{aligned} \label{eq:X2mom}
\EE[X^2_v| \mathcal{E}= \mathbf{e}]\; & \leq C  \xi(v)^2\Big( \!\!\!\sum_{\heap{a,b\in \mathrm{active}( \mathbf{e}):}{a,b>v, \; a\neq b}} \!\!\!\sfrac{(\log\frac av \vee 1)^\alpha(\log\frac bv \vee 1)^\alpha}{v\sqrt{ab}} + \!\!\!\sum_{\heap{a\in \mathrm{active}( \mathbf{e}):}{a>v}} \!\!\! \sfrac{(\log\frac av \vee 1)^\alpha}{\sqrt{va}} \Big).
\end{aligned}
\end{equation}
\item There are constants $0<c,C<\infty$ depending only on $f$ and $k$, such that for all $a\in \mathrm{active}( \mathbf{e})$
\begin{equation}\label{eq:Y1mom}
\EE[Y_a| \mathcal{E}= \mathbf{e}]\;\geq\; c  \frac{\xi(a)}{N}\sum_{\heap{v\in V:}{v>a}}\xi(v)^2 (\log\sfrac va \vee 1)^\alpha
\end{equation}
and 
\begin{equation}\label{eq:Y2mom}
\begin{aligned}
\EE[Y^2_a| \mathcal{E}= \mathbf{e}]\; \leq & \;C\Big(  \!\!\!\!\!\!\sum_{\heap{v, w\in V:}{v,w>a,\;v\neq w}} \!\!\! \!\!\!\xi(v)\xi(w)\sfrac{(\log\frac va \vee 1)^\alpha(\log\frac wa \vee 1)^\alpha}{a\sqrt{vw}} +  \!\sum_{\heap{v\in V:}{v>a}}\xi(v)^2\sfrac{(\log\frac va \vee 1)^\alpha}{\sqrt{va}} \Big).
\end{aligned}
\end{equation}
\end{enumerate}
\end{prop}
\begin{proof}
As $X_v$ is a constant multiple of a sum of indicators, its first conditional moment is 
\begin{align*}
\xi(v)\sum_{a\in \mathrm{active}( \mathbf{e})} & \PP(\Delta\Zz[v,a-1]=1| \mathcal{E}= \mathbf{e})\geq \frac12\xi(v) \!\!\!\sum_{a\in \mathrm{active}( \mathbf{e})} \!\!\!\PP(\Delta\Zz[v,a-1]=1)\\
&\geq \frac{c_{\ref{prop:expbounds2}}}{2}\xi(v) \!\!\!\sum_{\heap{a\in \mathrm{active}( \mathbf{e}):}{a>v}} \!\!\!\frac{(\log\frac av\vee 1)^\alpha}{\sqrt{va}}
\geq c\frac{\xi(v)^2}{N}\sum_{\heap{a\in \mathrm{active}( \mathbf{e}):}{a>v}}\xi(a)(\log\sfrac av\vee 1)^\alpha,
\end{align*}
where we have used \cref{lem:lbCond,lem:DM11Lemma2.9}, \cref{prop:expbounds2}, \eqref{eq:xitoroot} and chosen some appropiate constant $c>0$. A similar calculation for the second moment relies on \cref{lem:Mod212,lem:DM11Lemma2.9,lem:DM11Lemma2.10} and \cref{prop:expbounds1} 
and reads
\begin{align*}
&\xi(v)^2\sum_{a,b\in \mathrm{active}( \mathbf{e})}\PP(\Delta\Zz[v,a-1]=\Delta\Zz[v,b-1]=1| \mathcal{E}= \mathbf{e}) \\
& = \xi(v)^2\sum_{a,b\in \mathrm{active}( \mathcal{E})}\big(\PP(\Delta\Zz[v,a\vee b-1]=1| \mathcal{E}= \mathbf{e}, \Delta\Zz[v,a\wedge b-1]=1)\\
&\hspace{2cm}\times \PP(\Delta\Zz[v,a\wedge b-1]=1| \mathcal{E}= \mathbf{e})\big)\\
&\leq C^2_{\ref{lem:Mod212}}\xi(v)^2\sum_{a,b\in \mathrm{active}( \mathbf{e})}\PP^1(\Delta\Zz[v,a\vee b-1]=1)^{\1\{a\neq b\}}\PP(\Delta\Zz[v,a\wedge b-1]=1)\\
&\leq C^2_{\ref{lem:Mod212}}C_{\ref{prop:expbounds1}}^2\xi(v)^2\Big(\sum_{\heap{a,b\in \mathrm{active}( \mathbf{e}):}{a,b>v, \; a\neq b}}\frac{(\log\frac av \vee 1)^\alpha(\log\frac bv \vee 1)^\alpha}{v\sqrt{ab}} + \sum_{\heap{a\in \mathrm{active}( \mathbf{e}):}{a>v}}\frac{(\log\frac av \vee 1)^\alpha}{\sqrt{va}} \Big),
\end{align*}
This establishes~(i). Turning to (ii) we obtain firstly, for some appropriately chosen $c>0$,
\begin{align*}
\EE[Y_a| \mathcal{E}= \mathbf{e}]&=\sum_{v\in V}\xi(v)\PP(\Delta\Zz[a,v-1]=1| \mathcal{E}= \mathbf{e})\\
&\geq \sfrac{1}{2}\sum_{v\in V}\xi(v)\PP(\Delta\Zz[a,v-1]=1)
\geq c \, \xi(a)\sum_{v\in V:v>a}\frac 1v (\log\sfrac va \vee 1)^\alpha,
\end{align*}
where we have used \cref{lem:lbCond,lem:DM11Lemma2.9} for the first inequality and \cref{prop:expbounds2} and \eqref{eq:xitoroot} for the second. Secondly, analogous to the second moment calculation for~(i) we get
\begin{align*}
& \EE[Y^2_a|  \mathcal{E}= \mathbf{e}] \\
& = \sum_{v,w \in V}\xi(v)\xi(w)\PP(\Delta\Zz[a,v-1]=\Delta\Zz[a,w-1]=1| \mathcal{E}= \mathbf{e}) \\
& = \sum_{v,w\in V}\xi(v)\xi(w)\big(\PP(\Delta\Zz[a,v\vee w-1]=1| \mathcal{E}= \mathbf{e}, \Delta\Zz[a,v\wedge w-1]=1)\\
&\hspace{3cm}\times \PP(\Delta\Zz[a,v\wedge w-1]=1| \mathcal{E}= \mathbf{e})\big)\\
&\leq C^2_{\ref{lem:Mod212}}\sum_{v,w\in V}\xi(v)\xi(w)\PP^2(\Delta\Zz[a,v\vee w-1]=1)^{\1{\{v\neq w\}}}\PP^1(\Delta\Zz[a,v\wedge w-1]=1)\\
&\leq C^2_{\ref{lem:Mod212}}C_{\ref{prop:expbounds1}}^2\Big(\!\!\!\!\!\!\sum_{\heap{v, w\in V:}{v,w>a,\;v\neq w}}\xi(v)\xi(w)\sfrac{(\log\frac va \vee 1)^\alpha(\log\frac wa \vee 1)^\alpha}{a\sqrt{vw}} + \sum_{\heap{v\in V}{v>a}}\xi(v)^2\sfrac{(\log\frac va \vee 1)^\alpha}{\sqrt{va}} \Big),
\end{align*}
and the claim follows.
\end{proof}
The lower bounds \eqref{eq:X1mom} and \eqref{eq:Y1mom} now imply, that for $ \mathbf{e},V$ chosen as before
\begin{align}
&\EE[S^<(V)+S^>(V)| \mathcal{E} =  \mathbf{e}]\notag\\
& \geq \sfrac{c_{\eqref{eq:X1mom}}\wedge c_{\eqref{eq:Y1mom}}}{N} \!\!\!\!\! \sum_{\heap{a\in\mathrm{active}( \mathbf{e})}{v\in V}}\!\!\!\!\!\1\{v<a\}\xi(v)^2\xi(a)\big(\log\sfrac av \vee 1\big)^\alpha+\1\{v>a\}\xi(v)^2\xi(a)\big(\log\sfrac va \vee 1\big)^\alpha\notag\\
& \geq c \sum_{a\in\mathrm{active}( \mathbf{e})}\xi(a)\sum_{v\in V}\frac1v \big(\log\sfrac{a\vee v}{a\wedge v} \vee 1\big)^\alpha,\label{eq:FMLB}\end{align}
for some small $c>0.$ The factor $\sum_{v\in V}v^{-1} \big(\log(({a\vee v})/({a\wedge v})) \vee 1\big)^\alpha$ in the last sum is large as long as the set $V$ is sufficiently dense in $[N]$. In fact, the following instance of the pigeonhole principle applies, which is proved as \cref{lem:tedious2} in \cref{knownresults}.
\begin{lem}\label[lem]{lem:tediousmain}
There are $\eta\in(0,1)$ and $c>0$ only depending on $\alpha$ such that for any choice of $A\subset \{\lceil 2\e^2\rceil ,\dots, N \}$ and $v_0<\frac{\min A}{\e^2}\wedge\eta N$ satisfying
\begin{equation}\label{eq:conccond1}
\big(\log\sfrac{N}{\min A}\vee 1\big)^\alpha \xi^2(A)\leq \frac c2 N \big(\log\sfrac{N}{v_0}\big)^{\alpha+1},
\end{equation}
we have, for $V=\{v_0,\dots,N\}\setminus A$ and any $a\in A $,
\begin{equation}\label{eq:pigconclmain}
\sum_{v\in V}\frac1v \big(\log\sfrac{a\vee v}{a\wedge v}\vee 1\big)^{\alpha}\geq \frac c2 \big(\log\sfrac{N}{v_0}\big)^{\alpha+1},
\end{equation}
if $N$ is suffciently large.
\end{lem}

We summarise our observations in the following lemma. 

\begin{lem}[{Concentration of score}]\label[lem]{lem:concscore}
Let $\mathbf{e}$ be a proper configuration with $a_0=\min(\mathrm{active}( \mathbf{e}))$ and let $v_0<({a_0}/{\e^2})\wedge \eta_{\ref{lem:tediousmain}}N$ such that $V=\{v_0,\dots,N\}\cap \mathrm{veiled}(\mathbf{e})$ satisfies both \eqref{eq:colbd} for $k=2$ and \eqref{eq:conccond1} for $A=\mathrm{active}( \mathbf{e}).$
Then there exists a constant $c>0$ such that, for all $\beta\in(0,1)$,
\begin{align*}
\PP\big(S(V)\leq & \;(1-\beta)  \sfrac{c_{\eqref{eq:FMLB}}c_{\ref{lem:tediousmain}}}{2}\big(\log\sfrac N{v_0}\big)^{\alpha+1}\xi(\mathrm{active}( \mathbf{e})) \big| \mathcal{E}= \mathbf{e}\big)\\
\leq & \,\exp\Big( -\beta^2c\min\Big\{\sfrac{\xi(\mathrm{active}( \mathbf{e}))^2}{\xi^2(\mathrm{active}( \mathbf{e}))},\;\sfrac{\xi(\mathrm{active}( \mathbf{e}))(\log(N/{v_0}))^{\alpha+2}}{\xi(v_0)},\; v_0\big(\log \sfrac{N}{v_0}\big)^2\Big\} \Big)\\
& + \,C_{\ref{prop:colbd}} \big( \log\sfrac{N}{v_0}\vee 1\big)^{2\alpha+1}\frac{\xi(\mathrm{active}( \mathbf{e}))^2-\xi^2(\mathrm{active}( \mathbf{e}))}{N}.
\end{align*}
\end{lem}
\begin{proof}
 On the complement of the event $\mathcal{A}:=A_1(V, \mathcal{E})\cup A_2(V, \mathcal{E})\cup A_3(V, \mathcal{E})$ we choose $$\lambda=\beta\sfrac{c_{\eqref{eq:FMLB}}c_{\ref{lem:tediousmain}}}{2}\big(\log\sfrac N{v_0}\big)^{\alpha+1}\xi(\mathrm{active}(\mathbf{e})),$$ set $d={c_{\eqref{eq:FMLB}}c_{\ref{lem:tediousmain}}}/{2},$ and note that by \eqref{eq:FMLB},
 \begin{align*}
 \PP\big(\{S(V)\leq (1-\beta) & d\big(\log\sfrac N{v_0}\big)^{\alpha+1}\xi(\mathrm{active}( \mathbf{e}))\}\cap\mathcal{A} \big| \mathcal{E}= \mathbf{e}\big)\\
 \leq\; & \PP\big(\{S(V)\leq \EE[S(V)| \mathcal{E}= \mathbf{e}] -\beta d\big(\log\sfrac N{v_0}\big)^{\alpha+1}\xi(\mathrm{active}( \mathbf{e}))\}\cap\mathcal{A} \big| \mathcal{E}= \mathbf{e}\big).
 \end{align*}
  Applying \cref{lem:CL27}, we obtain
 \begin{equation}
 \label{eq:4help1}
 \begin{aligned}
 \PP\big(\{S(V)\leq (1-\beta) & d\big(\log\sfrac N{v_0}\big)^{\alpha+1}\xi(\mathrm{active}( \mathbf{e}))\}\cap\mathcal{A} \big| \mathcal{E}= \mathbf{e}\big)\\
 \leq & \exp\Big(-\frac{\beta^2d^2}{2}\frac{(\log\sfrac{N}{v_0})^{2\alpha+2}\xi(\mathrm{active}( \mathbf{e}))^2}{\sum_{v\in V}\EE[X_v|\mathcal{E}= \mathbf{e}]+\sum_{a\in \mathrm{active}( \mathbf{e})}\EE[Y_a|\mathcal{E}= \mathbf{e}]}\Big)
 \end{aligned}
 \end{equation}
and by \eqref{eq:X2mom} and \eqref{eq:Y2mom}, the sum in the denominator can be bounded
\begin{equation}
\label{eq:4help2}
\begin{aligned}\sum_{v\in V}&\EE[X_v|\mathcal{E}= \mathbf{e}]+\sum_{a\in \mathrm{active}( \mathbf{e})}\EE[Y_a|\mathcal{E}= \mathbf{e}]\\
 \leq\, & \, C_{\eqref{eq:X2mom}}  \sum_{v\in V}\xi(v)^2\Big( \!\!\!\sum_{\heap{a,b\in \mathrm{active}( \mathbf{e}):}{a,b>v, \; a\neq b}} \!\!\!\sfrac{(\log\frac av \vee 1)^\alpha(\log\frac bv \vee 1)^\alpha}{v\sqrt{ab}} + \!\!\!\sum_{\heap{a\in \mathrm{active}( \mathbf{e}):}{a>v}} \!\!\! \sfrac{(\log\frac av \vee 1)^\alpha}{\sqrt{va}} \Big)\\
 & + C_{\eqref{eq:Y2mom}}\sum_{a\in \mathrm{active}( \mathbf{e})}\Big(  \!\!\!\!\!\!\sum_{\heap{v, w\in V:}{v,w>a,\;v\neq w}} \!\!\! \!\!\!\xi(v)\xi(w)\sfrac{(\log\frac va \vee 1)^\alpha(\log\frac wa \vee 1)^\alpha}{a\sqrt{vw}} +  \!\sum_{\heap{v\in V:}{v>a}}\xi(v)^2\sfrac{(\log\frac va \vee 1)^\alpha}{\sqrt{va}} \Big).
\end{aligned}
\end{equation} 
 We now calculate bounds for all the terms appearing on the right hand side of \eqref{eq:4help2}. Observe that, for some appropiately chosen constant $D_1>0$, by \eqref{eq:xitoroot}
 $$\rho_1(v):=\sum_{\heap{a,b\in \mathrm{active}( \mathbf{e}):}{a,b>v, \; a\neq b}} \!\!\!\sfrac{(\log\frac av \vee 1)^\alpha(\log\frac bv \vee 1)^\alpha}{v\sqrt{ab}}\leq \big(\log\sfrac{N}{v_0}\big)^{2\alpha}\frac{D_1}v \sum_{\heap{a,b\in \mathrm{active}( \mathbf{e}):}{a,b>v, \; a\neq b}}\frac{\xi(a)\xi(b)}{N}$$
  and similarly, for some $D_2>0$,
  $$\rho_2(v):=\!\!\!\sum_{\heap{a\in \mathrm{active}( \mathbf{e}):}{a>v}} \!\!\! \sfrac{(\log\frac av \vee 1)^\alpha}{\sqrt{va}} \leq D_2 \big(\log\sfrac{N}{v_0}\big)^\alpha \frac{ \xi(v)\xi(\mathrm{active}( \mathbf{e}))}{N}.$$
  Combining the two estimates just obtained repeated use of \eqref{eq:xitoroot} yields \begin{equation}\label{eq:p16_1}\begin{aligned}
  \sum_{v\in V}\xi(v)^2\rho_1(v) +& \sum_{v\in V}\xi(v)^2\rho_2(v) \\ \leq & \; D \max\Big\{\frac1{v_0}\big(\log\sfrac{N}{v_0}\big)^{2\alpha}\xi(\mathrm{active}( \mathbf{e}))^2 ,\big(\log\sfrac{N}{v_0}\big)^{\alpha}\xi(v_0)\xi(\mathrm{active}( \mathbf{e}))\Big\}
 \end{aligned}
 \end{equation}
 for some $D>0,$ using $\sum_{v\in V}\xi(v)^2/v=O(N\sum_{v=v_0}^N v^{-2})=O(N/v_0)$ for the first sum, $\sum_{v\in V}\xi(v)^3=O(N^{3/2}\sum_{v=v_0}^N v^{-3/2})=O(N^{3/2}/{v_0}^{1/2})$ for the second sum, and finally $x+y\leq 2 (x\vee y)$. Next, we obtain in a similar way, for some $D_3>0$, $$\rho_3(a):=\sum_{\heap{v, w\in V:}{v,w>a,\;v\neq w}} \!\!\! \!\!\!\xi(v)\xi(w)\sfrac{(\log\frac va \vee 1)^\alpha(\log\frac wa \vee 1)^\alpha}{a\sqrt{vw}}\leq \; D_3\big(\log\sfrac{N}{v_0}\big)^{2\alpha}\frac{N}{a}\sum_{\heap{v, w\in V:}{v,w>a,\;v\neq w}}\frac{1}{vw} $$ and for some $D_4>0$ $$\rho_4(a):=\!\sum_{\heap{v\in V:}{v>a}}\xi(v)^2\sfrac{(\log\frac va \vee 1)^\alpha}{\sqrt{va}}\leq\;D_4 \big(\log\sfrac{N}{v_0}\big)^{\alpha} \frac{N}{\sqrt{a}}\sum_{\heap{v\in V:}{v>a}}v^{-\frac32}.$$ Consequently, mirroring the derivation of \eqref{eq:p16_1}, we obtain
 \begin{equation}\label{eq:p16_2}\begin{aligned}
 \sum_{a\in \mathrm{active}( \mathbf{e})}\rho_3(a)+ & \sum_{a\in \mathrm{active}( \mathbf{e})}\rho_4(a)\\ \leq & \; D' \max\big\{\big(\log\sfrac{N}{v_0}\big)^{2\alpha+2}\xi^2(\mathrm{active}( \mathbf{e})),\big(\log\sfrac{N}{v_0}\big)^{\alpha}\xi(v_0)\xi(\mathrm{active}( \mathbf{e}))\big\},
 \end{aligned}
 \end{equation}
 for some $D'>0$. Applying \eqref{eq:p16_1} and \eqref{eq:p16_2} in \eqref{eq:4help2} yields a bound on the denominator in \eqref{eq:4help1} from which the exponential term in the conclusion of the lemma is obtained. To conclude the proof it remains to note that the second term in the conclusion of the lemma is the bound on the probability of the occurence of $A_1(V, \mathcal{E})\cup A_2(V, \mathcal{E})\cup A_3(V, \mathcal{E})$ obtained in \cref{prop:colbd}.
\end{proof}
As a consequence of \cref{lem:concscore} we are able to bound the growth of the score from below as long as the total score of the explored vertices is not too large. To this end define, for given $s_0>0$ and $\delta_0\in(0,\sfrac12)$,
\begin{equation}\label{def:ell}
\ell_k :=\max\left\{n\in[N]: \sqrt{\frac{N}{n}}\geq\frac{ s_0 \delta_0}{1\vee \log k}\prod_{i=1}^{k-1}\sfrac{c_{\eqref{eq:FMLB}}c_{\ref{lem:tediousmain}}}{4}\big(\log\sfrac N{\ell_i}\big)^{\alpha+1} \right\}\; \textrm{ for all } k\geq 1.  
\end{equation}
If the maximum in the above definition is taken over the empty set, we let $\ell_k=1$.

\begin{rmk} Note that $\ell_k$ is defined in such a way that, up to a factor of order $\log k$, $(\xi(\ell_{k}))_{k\geq 1}$ mimics the superexponential growth of $(S_k)_{k\geq 1}$ described in \cref{lem:concscore}. It is precisely this property of $(\ell_k)_{k\geq 1}$ which makes them the correct truncation points for our exploration.
\end{rmk}	
Denoting by $K_*:=K_*(N)$ the first index $k$ for which $\ell_{k+1}=\ell_{k}$, we check that $(\ell_k)_{k\geq 1}$ satisfies the following decay condition.
\begin{lem}\label[lem]{lem:elldecay}
For any $\alpha\geq0,\delta\in(0,2\alpha+2)$ let $$k_0(\delta,\alpha)=\min\big\{k \geq 3 :\delta \log k \geq (2\alpha+2-\delta)k\log\big(1+\sfrac1k\big)+1 \big\}$$
then $$\ell_k\leq N\e^{-(2\alpha+2-\delta)(k-k_0)\log k}\;\textrm{ for all } k_0\leq k <K_*(N).$$
Also, there is a constant $c>0$ depending only on $s_0$ such that $$\ell_k\geq cN\e^{-(4\alpha+5)k(1\vee\log k)},\; \textrm{ for all }k.$$
\end{lem}
A verification of \cref{lem:elldecay} is provided in \cref{knownresults}, see \cref{lem:tedious3}. We conclude this section with the central result on the growth of the score in the truncated exploration. While \cref{lem:concscore} states that with high probability the total score of the active vertices grows by a factor close to $(\log N)^{\alpha+1}$ in every exploration step, the next proposition states that with high probability we may iterate the estimate of \cref{lem:concscore} and indeed reach a large score after $O(\log N / \log\log N)$ stages. 
\begin{prop}[{Score growth}]\label[prop]{prop:scoregrowth}
Let $\varepsilon,\eta>0$ and set $$K=\Big\lceil\Big(\frac{1}{2\alpha+2}+\eta\Big)\frac{\log N}{\log\log N}\Big\rceil.$$ Then there are $s_0(\varepsilon)>0,\delta_0(\varepsilon)\in(0,\sfrac12)$ and $N_0(\varepsilon,\eta)$ such that
$$\PP\Big(\xi(\mathrm{active}(\mathcal{E}_K) \cup \mathrm{dead}(\mathcal{E}_K) )\leq \frac{\sqrt{N}}{(\log N)^{\alpha+1}}\Big)\leq \varepsilon,\; \textrm{ for all }N\geq N_0,$$
where $(\Ee_{k})_{k\geq 0}$ is the exploration in $\Gg_N$ with truncation $(\ell_k)_{k\geq 1}$ as in \eqref{def:ell} that is started in a proper configuration $\Ee_0$ satisfying ${\xi(\mathrm{active}(\mathcal{E}_0))}/{\xi(\min(\mathrm{active}(\mathcal{E}_0)))}\geq s_0.$
\end{prop}
\begin{proof}
We first note that, for fixed $\eta>0$, $K_*<K$ for all sufficiently large $N$ by the first statement in \cref{lem:elldecay}. We wish to iteratively apply \cref{lem:concscore} until $k\leq K_*(N)$ is so large that the second conclusion of \cref{lem:elldecay} allows us to establish the lower bound for $\xi(\mathrm{active}(\mathcal{E}_K) \cup \mathrm{dead}(\mathcal{E}_K))$. To this end let, for $k\geq 0,$ $$S_k:=\xi(\mathrm{active}(\mathcal{E}_k)),\quad H_k:=\xi(\mathrm{active}(\mathcal{E}_k) \cup \mathrm{dead}(\mathcal{E}_k)),  \quad a_k:=\min(\mathrm{active}(\mathcal{E}_k)),$$ 
and furthermore $$K_0:=K_0(N):=\min\Big\{k: H_k > \frac{\sqrt{N}}{(\log N)^{\alpha+1}}\Big\}.$$ To accomplish this, we need to bound the total probability of error which arises by repeatedly applying \cref{lem:concscore}. The proof is complete once we have verified the following three claims:

\begin{enumerate}[$(i)$]
\item For given $\varepsilon,\delta_0$ we may choose $s_0>0$ such that, for all sufficiently large $N$,
the configuration $\Ee_0$ and $\ell_1$ satisfy the conditions of \cref{lem:concscore} unless $K_0=0.$ Additionally, with probability exceeding $1-\gamma_1$, for \smash{$\gamma_1:={6\varepsilon}/({2\pi^2})$} we have, for all sufficiently large $N$, 
$$S_1> \sfrac{c_{\eqref{eq:FMLB}}c_{\ref{lem:tediousmain}}}{4}\big(\log\sfrac N{\ell_1}\big)^{\alpha+1}S_0.$$
\item Conditional on $k< K_0\wedge K_*$ and $\mathcal{E}_j,1\leq j\leq k$, satisfying 
$$S_j> \sfrac{c_{\eqref{eq:FMLB}}c_{\ref{lem:tediousmain}}}{4}\big(\log\sfrac N{\ell_j}\big)^{\alpha+1}S_{j-1},$$ the configuration $\Ee_k$ and $\ell_{k+1}$ satisfy the conditions of \cref{lem:concscore}. Consequently, we can find $\gamma_k>0$, such that with conditional probability exceeding $1-\gamma_k$ we have
$$ S_{k+1}> \sfrac{c_{\eqref{eq:FMLB}}c_{\ref{lem:tediousmain}}}{4}\big(\log\sfrac N{\ell_{k+1}}\big)^{\alpha+1}S_{k},$$ and thus $K_*\geq K_0\geq k+1.$
\item $\delta_0=\delta_0(\varepsilon)$ may be fixed in such a way that $(\gamma_k)_{k\geq 1}$ from (i) and (ii) satisfies $\sum_{k=1}^{L}\gamma_k<{\varepsilon}$ as $N\to\infty$ for any $L=O(\log N(\log\log N)^{-1})$.
\end{enumerate}
Note that $(\gamma_k)_{k\geq 1}$ serves as a proxy for the probability that in exploration step $k+1$ the exploration process violates the conditions of \cref{lem:concscore}.
\medskip

\emph{Proof of (i):} If $K_0=0$, then there is nothing to show. Let $K_0>0.$ Given $\delta_0$, the condition $\ell_1<{a_0}/{\e^2}$ is satisfied by choosing $s_0$ sufficiently large. The conditions $\ell_1<\eta_{\ref{lem:tediousmain}}N$, \eqref{eq:colbd} and \eqref{eq:conccond1} are now implicit in the assumption $K_0>0$, for all sufficiently large $N$. Application of \cref{lem:concscore} with $\beta=\frac12$ yields $$\PP\Big(S_1\leq \sfrac{c_{\eqref{eq:FMLB}}c_{\ref{lem:tediousmain}}}{4}\big(\log\sfrac N{\ell_1}\big)^{\alpha+1}\xi(\mathrm{active}( {\mathcal{E}_0}))\Big)\leq \e^{-\frac{c_{\ref{lem:concscore}}}{4}s_0}+o(1),$$ as $N\to\infty,$ thus, after possibly increasing $s_0$ again, $(i)$ holds.\medskip

\emph{Proof of (ii):} Assume that $K_*>K_0>k$ and note that this implies $\xi(\ell_{k+1})\leq \sqrt{N}$. By definition of the exploration we have $a_k\geq \ell_k$. By \cref{lem:elldecay} and the definition of
 $(\ell_k)_{k\geq 1}$, the network size
$N$ can be  chosen so large that $(\ell_k)_{k\geq 1}$ decays faster than $(\e^{-2k})_{k\geq 1}$  for all $k<K_*(N)$. In particular $\ell_{k+1}<{a_k}/{\e^2}$ holds. As in the proof of $(i)$, $K_0>k$ implies that \eqref{eq:conccond1} is satisfied and also, using $\xi(\ell_{k+1})\leq\sqrt{N},$ \eqref{eq:colbd} must hold. Hence we may again apply \cref{lem:concscore} with $\beta=\frac12$ to obtain that, conditionally on $\mathcal{E}_k$,
\begin{align} &\PP\Big(S_{k+1}\leq \sfrac{c_{\eqref{eq:FMLB}}c_{\ref{lem:tediousmain}}}{4}\big(\log\sfrac N{\ell_{k+1}}\big)^{\alpha+1}S_{k}\Big) \notag\\
&\leq \;\exp\big(-\sfrac{c_{\ref{lem:concscore}}}{4}\min\big\{\sfrac{S_k^2}{\xi^2(\mathrm{active}(\mathcal{E}_k))},\sfrac{S_k(\log(N{\ell^{-1}_{k+1}}))^{\alpha+2}}{\xi(\ell_{k+1})}, \ell_{k+1}(\log (N{\ell^{-1}_{k+1}}))^{2}\big\}\big)\notag\\ & \phantom{averyveryverylongthing} + C_{\ref{prop:colbd}} \big( \log\sfrac{N}{\ell_{k+1}}\vee 1\big)^{2\alpha+1}\frac{S_k^2}{N}\notag\\
&=:\Delta_k+\Gamma_k.\label{eq:finalbound}
\end{align}
The conclusion $K_*\geq K_0\geq k+1$ holds if $$S_{k+1}> \sfrac{c_{\eqref{eq:FMLB}}c_{\ref{lem:tediousmain}}}{4}\big(\log\sfrac N{\ell_{k+1}}\big)^{\alpha+1}S_{k},$$ as ${S_{k+1}}/{S_k}\geq {\xi(\ell_{k+1})}/{\xi(\ell_{k})},$ by choice of the defining recursion \eqref{def:ell} for the truncation $(\ell_{k})_{k\geq 1}$.\\
\\
\emph{Proof of (iii):} It remains to bound the random terms $\Delta_k+\Gamma_k, k\leq K_0,$ appearing in \eqref{eq:finalbound} by some deterministic sequence $\gamma_k$ with the desired summability property. We start with~$\Gamma_k$. Since $S_k< H_{K_0},$ we get $$\Gamma_k\leq C_{\ref{prop:colbd}}(\log N)^{2\alpha+1}\frac{S_k^2}{N}\leq \frac{C_{\ref{prop:colbd}}}{\log N},$$ thus $\sum_{k=1}^{L} \Gamma_k=O((\log\log N)^{-1})$ for $L=O(\log N(\log\log N)^{-1}).$ To bound $\Delta_k$, we analyse the three terms under minimisation separately. Since $x\mapsto x(\log({N}/{x}))^2$ is strictly increasing on $[1,{N}/{\e^2}]$, the deterministic rightmost term satisfies $$\ell_{k+1}\big(\log ({N}/{\ell_{k+1}})\big)^{2}\geq (\log N)^2.$$ By definition of $(\ell_k)_{k\geq 1}$ and \eqref{eq:xitoroot} there is some constant $c$, independent of $k,N$ and $\varepsilon$, such that 
\smash{$\xi(\ell_{k+1})\leq c (\log({N}/{\ell_{k}}))^{\alpha+1}\xi(\ell_k)$} and thus $$\frac{S_k(\log\sfrac N{\ell_{k+1}})^{\alpha+2}}{\xi(\ell_{k+1})}\geq\frac{S_k(\log\sfrac{N}{\ell_{k+1}})^{\alpha+2}}{c \big(\log\sfrac{N}{\ell_{k}}\big)^{\alpha+1}\xi(\ell_{k})}\geq \frac{S_k}{c\xi(\ell_k)}.$$
Since $$\sum_{u\in\mathrm{active}(\Ee_k)}\xi^2(u)\leq \xi(a_k)S_k\leq \xi(\ell_k)S_k,$$ we also have $$\frac{S_k^2}{\xi^2(\mathrm{active}(\mathcal{E}_k))}\geq\frac{S_k}{\xi(\ell_k)}.$$ On the conditioning event of $(ii)$, we have \begin{equation}
\label{eq:S_vs_ell}
S_k\geq s_0\prod_{i=1}^{k}\frac{c_{\eqref{eq:FMLB}}c_{\ref{lem:tediousmain}}}{4}\big(\log\sfrac N{\ell_{i}}\big)^{\alpha+1}
\end{equation} 
and thus $$\frac{S_k}{\xi({\ell_k})}\geq (1\vee \log k)\frac{c'}{\delta_0},$$
for some constant $c'>0$. Combining all estimates we obtain, for some $c''>0$, $$\Delta_k\leq\e^{-c''\big(\sfrac{(1\vee \log k)}{\delta_0}\wedge (\log N)^2\big)}.$$ This implies that by choosing $\delta_0$ small enough we may obtain $\Delta_k\leq \big({6\varepsilon}/({2\pi^2 k^2})\big)\vee N^{-c''\log N}$ and the last claim is proved.
\end{proof}

\subsection{Connectivity of high degree vertices}
{
We now provide a connectivity result for those vertices in $\Gg_N$ which have a very high degree. This sprinkling-type argument is close in spirit to the proof of a diameter result for the `inner core' of a different preferential attachment model in \cite{dommers_diameters_2010}.
\medskip

Fix a sequence $(M_N)_{N\in\N}$ of positive integers satisfying $\log M_N=o(\log N)$. We will now define a random subset $C_N\subset[N]$ of 
size 
at most $M_N$ which has small diameter in $\Gg_N$. To this end, fix $\eE\in(0,1)$ and associate $N\in\N$ with $N_\eps= \lceil (1+\eE)^{-1} N\rceil$. Assuming that $N$ is sufficiently large such that $M_N\leq N_\eE$ we call the elements of the random set 
	\begin{equation}
	\label{def:CN}
	C_N=\{v=1,\dots,M_N: f(\Zz[v,N_\eE])\geq  \sfrac{1}{2}\EE f(\Zz[M_N,N_\eE])\}
	\end{equation}
	\emph{core vertices} of $\Gg_N$. We show below that the diameter of $C_N$ in the random graph $\Gg_N$ is bounded with high probability, but first we provide an estimate for the number of vertices in $C_N$. 
	
	\begin{lem}[Size of $\text{core}_N$]\label{lem:sizelemma}
		There exists a constant $c=c(\eps)>0$ such that
		$$
		\# C_N\geq c M_N, \textrm{ with high probability as }N\to\infty,
		$$
		where $C_N$ is as in \eqref{def:CN}.
	\end{lem}
	
	\begin{proof}
		Note that by the Paley-Zygmund inequality one has for  $v\in [M_N]$  
		\begin{align*}\PP\big(f(\Zz[v,N_\eE])\geq\sfrac{1}{2}\EE f(\Zz[v,N_\eE])\big)  
		\;\geq \;\frac{(\EE f(\Zz[v,N_\eE]))^2}{4\, \EE f(\Zz[v,N_\eE])^2}=:p(v,N).
		\end{align*}
		By Propositions~\ref{prop:expbounds1} and \ref{prop:expbounds2} there exists $p^*>0$ such that for large $N\in\N$, $p(v,N)>p^*$ for all $v\leq M_N$. Since further 
		the degree evolutions $(\Zz[v,\cdot]:v=1,\dots,M_N)$ are independent and $\EE f(\Zz[v,N_\eps])\geq \EE f(\Zz[M_N,N_\eps])$, for all $v\in[M_N]$, we conclude that
		$$
		\frac {\# C_N}{M_N}\geq p^*/2, \text{ with high probability.}
		$$
	\end{proof}

	\begin{prop}[Diameter of the core]\label{prop:cores} Let $C_N$ be as in \eqref{def:CN} and $M_N=\lfloor (\log N)^R \rfloor$ for some $R>0$. Then, with high probability as $N\to \infty$ we have $$\max_{u,v\in C_N} d_N(u,v)\leq \max\Big(6, \Big\lfloor\frac{R}{\alpha}\Big\rfloor+2\Big).$$
	\end{prop}
For the proof of \cref{prop:cores} we use \emph{multinomial random graphs}. This random graph model depends on three parameters: a finite set of vertices $\mathcal{V}$, an iteration number $t$ and a success probability $r\geq 0$ with
$$
r \frac{\#\mathcal{V}(\#\mathcal{V}-1)}2\leq 1.
$$
The corresponding multinomial random graph is an undirected multigraph that is constructed as follows. We denote by $\mathcal A(v,w)$ the random number of edges that connect two distinct vertices $v$ and $w$ of $\mathcal V$. An $\mathcal M(\mathcal V, t,r)$-graph $(\mathcal V,\mathcal A)$ is obtained by choosing
$$
(\mathcal A(v,w): v<w  \text{ distinct vertices of }\mathcal V)
$$
multinomially distributed with $t$ draws and identical success probabilities $r$. Note that we do not assume that $r \frac{\#\mathcal{V}(\#\mathcal{V}-1)}2=1$ which means that formally the random vector has to be extended by a dummy variable which gets the remaining mass. 

Recall that the sum of two independent multinomial random variables with identical success probabilities is again multinomial. Hence the sum of two independent multinomial random graphs with identical sets of vertices and success probabilities is a multinomial random graph with the same success probability with the number of draws being the sum of the two  draw parameters. We will make use of this fact in the proof of {\cref{prop:cores}} below.

\begin{lem}\label{lem:multinomial}
Let $(\mathcal{V},\mathcal{A})\sim\mathcal{M}(\mathcal V,t,r)$ with $rt\geq \#\mathcal{V}^{\rho-1}$ for some $\rho>0$. Then, with high probability as $t\to\infty$,  the diameter of $(\mathcal{V},\mathcal{A})$ is bounded by $\max(3,\lfloor1/\rho\rfloor+1)$.
\end{lem}
\begin{proof}
The detailed argument is given in Lemma A.2 and Proposition 3.2 of \cite{dommers_diameters_2010}. We give a brief outline here: First one shows that the diameter of $(\mathcal{V},\mathcal{A})$ is bounded by the diameter of the \emph{uniform random graph} $\Gg_{u}(\#\mathcal{V},m)$ with $\#\mathcal{V}$ vertices and $m=m(t)$ edges, where $$m(t)=\Big\lceil\frac{\#\mathcal{V}(\#\mathcal{V}-1)}{4}\big(1-(1-r)^t\big)\Big\rceil.$$ The graph $\Gg_{u}(\#\mathcal{V},m)$ is in turn asymptotically equivalent to the classical \emph{Erd\H{o}s-R\'{e}nyi graph} $\Gg(\#\mathcal{V},p)$ on $\#\mathcal{V}$ vertices with edge probability $p=p(t)$ given by
\begin{equation}
\label{eq:ERprobapprox}
p(t)=\frac12 \big(1-(1-r)^t\big).
\end{equation}
The well known diameter result for dense $\Gg(n,p)$, see e.g. \cite[Corollary 10.12]{bollobs_random_2001}, states that $\Gg(n,p)$ has with high probability diameter $d \geq 3$ if 
\begin{equation}\label{dia0}
\lim_{n\to \infty}\frac{\log n}{d}-3\log\log n=\infty,
\end{equation}
\begin{equation}\label{dia1}
\lim_{n\to \infty}p^dn^{d-1}-2\log n=\infty,
\end{equation}
and
\begin{equation}\label{dia2}
\lim_{n\to \infty}p^{d-1}n^{d-2}-2\log n=-\infty.
\end{equation}
Finally, by assumption, for $n=\#\mathcal{V}$, we have $rt\geq n^{\rho-1}$ and therefore by \eqref{eq:ERprobapprox} $$p= c(1+o(1)) n^{\rho-1},$$ for some constant $c$. We may assume $\rho\leq 1/2$, since clearly decreasing $\rho$ only increases the diameter. Setting $d=\lfloor1/\rho\rfloor+1$ it is obvious, that \eqref{dia0} holds and furthermore $$p^dn^{d-1}\approx n^{d\rho-1}=n^{\lfloor 1/\rho \rfloor\rho -1 +\rho}\quad{\textrm{ and }}\quad p^{d-1}n^{d-2}\approx n^{(d-1)\rho -1}=n^{\lfloor 1/\rho \rfloor\rho -1 },$$ implying that \eqref{dia1} and \eqref{dia2} are satisfied as well.
\end{proof}
\begin{proof}[Proof of {\cref{prop:cores}}]
	We use a coupling of $C_N\subset\Gg_N$ and a multinomial random graph to show that the diameter of $C_N$ is small. Recall that the preferential attachment model is uniquely specified by the degree evolutions which can be constructed as follows. Take a family of independent Uniform$[0,1]$ random variables 	$(U(v,n): v,n\in \mathbb N \text{ with } v<n)$ and define iteratively
	\begin{equation}\label{COUPLUZ}
\Zz[v,n]=0 \text{ and } \Zz[v,n] =\Zz[v,n-1]+\1 \Bigl\{U(v,n)\leq \frac{f(\Zz[v,n-1])}{n-1}\Bigr\},\text{ \ for }n=v+1,\dots
	\end{equation} 
	Let  $N\in\N$ and  $c_N\subset [N_\eps]$ such that
	\begin{equation}
	\label{CNspecs}
	\#c_N\to\infty \text{ and } \log(\#c_N)=o(\log N)
	\end{equation} and construct for each $n\in[N]\backslash[N_\eps]$ a multinomial random graph $(c_N,\mathcal A_n)$ with iteration number one by the rule that for distinct vertices $v,v'\in c_N$ the edge $( v,v')$ is present if and only if 
	\begin{equation}
	\label{edgepresent}   \{v,v'\} =\{  w  \in c_N :  U(w,n)\leq \EE f(\Zz[M_N,N_\eps])/(2N)\},
	\end{equation}
    thus the success probability equals
\begin{equation}
	\label{eq:pairconnectionprob}
	r(N):=\Big(\frac{\EE f(\Zz[M_N,N_\eps])}{2N}\Big)^2\Big(1-\frac{\EE f(\Zz[M_N,N_\eps])}{2N}\Big)^{\#c_N-2}.
	\end{equation}
	Clearly, $r(N)\in(0,1)$ if $N$ is sufficiently large by \eqref{CNspecs}. Note that the random graphs $(c_N,\mathcal A_{[N_\eps]+1}),\dots, (c_N,\mathcal A_{[N]})$ are independent and the sum of the latter graphs, say $(c_N,\mathcal A^{N})$, is a binomial random graph with iteration number $N-N_\eps$ and success probability $r(N)$. Furthermore, by \eqref{edgepresent} and \eqref{COUPLUZ}, for any $v,w\in c_N$ with $f(\Zz[v,N_\eps])\geq \EE f(\Zz[M_N,N_\eps])$ and $f(\Zz[w,N_\eps])\geq \EE f(\Zz[M_N,N_\eps])$ the existence of the edge $(v,w)$ in the multinomial graph $(c_N,\mathcal A_n)$ $(n=[N_\eps]+1,\dots,[N])$ implies the existence of edges $(v,n)$ and $(w,n)$ in the graph~$\Gg_N$. Thus the diameter of $c_N$ in $\Gg_N$ is less than twice the diameter of the multinomial random graph $(c_N,\mathcal A^N)$.
	\smallskip
	
	Next we show that for a sequence of sets $c_N\subset [N_\eps]$ satisfying $\delta M_N\leq \#c_N\leq M_N$, for some $\delta>0$, the random graphs $(c_N,\mathcal A^N)$ satisfy the assumptions of \cref{lem:multinomial}. By \cref{prop:expbounds1}, for some $C>0$, 
	\begin{align*}
	&\Big(1-\frac{\EE f(\Zz[M_N,N_\eps])}{2N}\Big)^{\#c_N-2}\geq \Big(1-\frac{C(\log N)^\alpha}{\sqrt{N}}\Big)^{\#c_N}\\
	&=\,1-\frac{C\#c_N(\log N)^\alpha}{\sqrt{N}}+O\big((\#c_N^2(\log N)^{2\alpha} N^{-1})\big),
	\end{align*}
	which converges to one as $\log(\#c_N) = o(\log N)$. Hence we obtain, using again \cref{prop:expbounds1}, that $$r(N)\geq  d\frac{(\log\sfrac{N_\eps}{M_N})^{2\alpha}\frac{N_\eps}{M_N}}{N^2}\geq d_\eps \frac{(\log N)^{2\alpha}}{N \#c_N},$$ for some suitably chosen constants $d,d_\eps>0.$ It now follows from $N-N_\eps\geq N \eps/2$, that $$r(N)(N-N_\eps)\geq \frac{\eps}{2} d_\eps \frac{(\log N)^{2\alpha}}{\#c_N}=\frac{\eps}{2} d_\eps {(\#c_N)}^{2\alpha\log\log N/\log \#c_N-1}.$$ Using that $\log \#c_N= R\log\log N+O(1)$, by choice of $M_N$ and $c_N$, we thus may apply \cref{lem:multinomial} for any $\rho<2\alpha/R,$ which yields a diameter bound of $\max(3,\lfloor R/(2\alpha)\rfloor+1)$ on $(c_N,\mathcal{A}^N)$ with high probability as $N\to\infty.$
	\smallskip
	
	Finally, note that $\{C_N=c_N\}$ and the random variables $\{U(v,n): v\in c_N, n\in[N]\backslash[N_\eps]\}$ are independent. Hence the event $\{C_N=c_N\}$ is independent of the realisation of $(c_N,\mathcal A^{N})$. By \cref{lem:sizelemma}, the conclusion of the last paragraph may thus be applied to $(C_N,\mathcal{A}^N)$ outside a set of vanishing probability and recalling that one edge in $(C_N,\mathcal{A})$ corresponds to two edges in $\Gg_N$ now yields the bound claimed in the proposition with high probability as $N\to\infty$.\end{proof}

\subsection{Proof of \cref{thm:PA}}
It remains to prove the upper bound by combining the results about the first two phases of the explorations of \emph{two} independently chosen vertices, and join the connected components uncovered during these explorations to $C_N$. 

\begin{proof}[Proof of \cref{thm:PA}]
We start local explorations in the uniformly chosen vertices $U,V$ from the largest connected component $\mathcal{C}_N\subset\Gg_N$. Let $\eps\in (0,1/3)$ be fixed. Since $\#(\mathcal{C}_N \setminus \mathcal{C}_{N_\eps})\leq \eps N$ for $N_\eps= \lceil (1+\eE)^{-1} N\rceil$ we have $U,V\in \Gg_{N_\eps}$ with probability exceeding $1-2\eps.$ We consider two exploration processes around $U$ and $V$, respectively, in $\Gg_{N_\eps}.$\smallskip

By \cref{prop:localknown} there exists $k_0(\varepsilon)$ such that with probability exceeding $1-\eE/4$, in both explorations we reach after at most $k\leq k_0(\varepsilon)$ exploration steps active sets $A\subset \mathcal{E}_k$ satisfying  $\xi(A)\geq s_0\xi(\min A)$ with $s_0=s_0({\varepsilon}/{8})$, as defined in \cref{prop:scoregrowth}. Now we start the main phase of the two explorations with initial configurations in which the sets $A$ represent the active vertices, and possible other active vertices are veiled and connecting edges removed. Observe that this modification can only increase the observed distance between $U$ and $V$.\smallskip

We denote the explored parts of the network at this stage by $\mathcal{E}_0^{\ssup{1}},\mathcal{E}_0^{\ssup{2}}$ and henceforth only look at the scores of the two explorations. To keep the explorations sufficiently independent, we slightly modify the algorithm: The exploration process around $U$ inspects for any active vertex $v$ only connections to $w>v$, if $w\in[N]$ is \emph{odd}. Similarly, the exploration around $V$ only checks an active vertex $v$ for connections to $w>v$ if $w\in[N]\}$ is an \emph{even} vertex. It is easily seen that this only changes the constant in the lower bound of \cref{lem:tediousmain}.\smallskip

We know by \cref{prop:scoregrowth} that if $N_\eps$ is sufficiently large, then for each exploration viewed on its own, with probability exceeding $1-\eE/4$, after 
$$K_0^{\ssup i}\leq \;\Big(\frac{1}{2\alpha+2}+\frac{\eta}{2}\Big) \frac{\log N}{\log\log N}$$
steps and any choice of $\eta>0$, the conclusion of \cref{prop:scoregrowth} is applicable. We call such an exploration \emph{successful}. In the step when the score bound in \cref{prop:scoregrowth} is reached we have
\begin{equation}\label{eq:scorebound2}
H_K^{\ssup{i}}\geq \frac{\sqrt{N_\eps}}{(\log N_\eps)^{1+\alpha}}, 
\end{equation}
where $$H_K^{\ssup{i}}:=\xi\Big(\mathrm{active}\big(\mathcal{E}_{^{K_0^{\ssup i}}}^{\ssup{i}}\big)\cup\mathrm{dead}\big(\mathcal{E}_{^{K_0^{\ssup i}}}^{\ssup{i}}\big),N_\eps\Big).$$ 
Also note for later reference that by the definition \eqref{def:ell} of $(\ell_k)_{k\geq 1}$ and the recursion for $(S_k)$, cf. \eqref{eq:S_vs_ell}, 
\begin{equation}\label{eq:scorebound3}
H_{K}^{\ssup{i}}\geq d\,\frac{\log N}{\log\log N}\,\xi(\ell_{K_0},N_\eps),
\end{equation}
for some small $d>0.$
\smallskip

We may assume without loss of generality that $K_0^{\ssup 1}<K_0^{\ssup 2}$. After stage $K_0^{\ssup 1}$, we cannot apply exactly the same reasoning for the second exploration as in \cref{prop:scoregrowth}, since the total score of both configurations combined is too high. However, the lower bound given in \cref{lem:lbCond} can still be applied in each exploration step, since the set $I_0$ of non-jump times featured in this lemma consists only of odd vertices and is therefore disjoint of the sets of non-jump times used in the other exploration which may have exceeded the score bounds. The restriction on the set of jump-times $I_1$ clearly plays no role -- if we encounter an additional jump due to a connection to the first exploration, then the procedure can be stopped and a shortest path connecting $U$ and $V$ is found.\smallskip

As a consequence, we deduce that with high probability, $U$ and $V$ are either found to be connected before stage $K_0^{\ssup 2}$ or their respective explorations have reached a score of at least $\sqrt{N_\eps}(\log N_\eps)^{-(\alpha+1)}.$ Note that for a successful exploration, by definition of $K_0$, $$ \sqrt{\frac{N_\eps}{(\log N_\eps)^{\alpha+2}}}\geq S_{K_0-1}$$ and furthermore \eqref{def:ell} and \eqref{eq:S_vs_ell} imply that $$S_{K_0-1}\geq \sqrt{\frac{N_\eps}{\ell_{K_0}}}.$$ Combining these estimates it follows that $\ell_{K_0}>(\log N_\eps)^{2\alpha+2}$ and the exploration has thus collected no information about the degree evolutions of vertices in $[M_N]$ during its main phase, where $M_N=\lfloor c_{\eps}(\log N)^{2\alpha+2}\rfloor$, and $c_\eps>0$ is some suitably chosen constant. Therefore we can apply \cref{lem:sizelemma,prop:cores} to deduce that, for sufficiently large $N$, with probability exceeding $1-\eE/4$, the subgraph induced by $C_N\subset\Gg_N$ is of bounded diameter $D$ and contains at least $r M_N$ vertices, for some $r=r(\varepsilon)>0$. \smallskip

Denoting the sets of active and dead vertices of \smash{$\mathcal{E}^{\ssup{i}}_{^{K_0^{(i)}}}$} by $V(i)$, and using the shorthand $$\{V(i)\con{\eps} C_N\}:=\{\exists\; n\in [N]\setminus[N_\eps], v\in V(i) , w\in C_N: n\to v, n\to w\},\quad i=1,2,$$ it remains to show that $$\PP(V(1)\con{\eps} C_N,\,V(2)\con{\eps} C_N)\geq 1-\eE/4,$$ if $N$ is sufficiently large. Conditional on $\Gg_{N_\eps}$, let $$L=\{j\in \{N_\eps+1,\dots,N\}:\; \exists\, v \in C_N \textrm{ with } j \to v\}.$$
We have already established that, with high probability, $C_N$ contains at least $rM_N$ vertices. It is now straightforward to deduce via an appropriate coupling to Bernoulli random variables that \begin{equation}
\label{LBOUND} \#L\geq  qM_N\psi(M_N,N_\eps)\xi(M_N,N_\eps)
\end{equation} 
with probability at least $1-{\varepsilon}/{12}$, where $q=q(\varepsilon)>0$ is some small constant. Each $j\in L$ has an independent probability of at least $f(\Zz[v,N_\eps])/N$ to connect to $v\in V(i)$, thus the probability that it does not connect to any $v\in V(i)$ is bounded above by {$\exp\big(-N^{-1}\sum_{v\in V(i)}f(\Zz[v,N_\eps])\big)$}. Since this holds independently for all $j\in L$, we obtain by \eqref{eq:scorebound2} and \eqref{LBOUND}, recalling that $\psi(M_N,N_\eps)\xi(M_N,N_\eps)\approx \log(N/M_N)^{\alpha}\sqrt{N/M_N}$, \begin{equation}\label{eq:condfinal}
\begin{aligned}
&\1\{\#L\geq  qM_N\psi(M_N,N_\eps)\xi(M_N,M_\eps) \} \\
&\times\1\big\{\textstyle\sum_{v\in V(i)}f(\Zz[v,N_\eps])\geq \nu \xi(V(i),N_\eps)\big\}\PP\big(\{V(i)\con{\eE}C_N \}^{c} \, \big| \, \Gg_{N_\eps}\big)\\
&\leq\;\exp\Big(-\frac{\#L\sum_{v\in V(i)}f(\Zz[v,N_\eps])}{N}\Big) \leq\; \exp\Big(-\frac{\#L \nu}{N} H_{K}^{\ssup{i}}\Big)\\
&\leq\; \exp\Big(-\frac{\nu (qM_N\psi(M_N,N_\eps)\xi(M_N,N_\eps)-1)\sqrt{N_\eps}}{ N \log N^{\alpha+1}}\Big) \leq\; \frac{\varepsilon}{24},
\end{aligned}
\end{equation}
for all sufficiently large $N$ and some small $\nu\in(0,1)$ to be fixed below. Note that the term in the last exponential is bounded below by a constant (depending only on $\eps$) multiple of $(\log(N/M_N))^{\alpha}$.\smallskip

It remains to fix $\nu>0$ and bound $$\PP\Big(\sum_{v\in V(i)}f(\Zz[v,N_\eps])< \nu \xi(V(i),N_\eps)\Big)=\PP\Big(\sum_{v\in V(i)}f(\Zz[v,N_\eps])< \nu H_{K}^{\ssup{i}}\Big).$$ The proof of \cref{prop:scoregrowth} shows that $(S_k)_{k=1}^{K_0}$ grows superexponentially, thus for every $\mu>0$, there is $\nu>0$ such that $$\sum_{v\in \mathrm{active}(\mathcal{E}_{K_0})}f(\Zz[v,N_\eps])\geq \mu S_{K_0}\Rightarrow \sum_{v\in V(i)}f(\Zz[v,N_\eps])\geq \nu H_{K_0},$$ i.e. $H_{K_0}$ can differ from $S_{K_0}$ by at most a constant factor. Therefore it is sufficient to find a lower bound on $S_{K_0}$. Note that, for $v\in \mathrm{active}(\mathcal{E}_{K_0})$, replacing the attachment rule $f$ by the linearised attachment rule $\bar{f}(k)=f(0)+{k}/{2}$ does not change the values $\xi(v,N_\eps)$ and only diminishes the sum on the left. For the rest of the argument we may therefore assume that $f=\bar{f}$ in the evolutions $\{\Zz[v,\cdot],v\in \mathrm{active}(\mathcal{E}_{K_0})\}$. During the final exploration stage $K_0$, the evolution $\Zz[v,i]_{i=1}^{N_\eps}$ of an active vertex $v$ is only conditioned on a set $I_0$ of non-jumps which still fullfills the conditions of \cref{lem:lbCond}. This implies that, for some small $s>0$, we have
$\EE[f(\Zz[v,N_\eps])|\mathcal{E}_{K_0}]\geq s\xi(v,N_\eps),$ and thus $$\EE\Big[\sum_{v\in \mathrm{active}(\mathcal{E}_{K_0})}f(\Zz[v,N_\eps])\Big|\mathcal{E}_{K_0}\Big]\geq s S_{K_0}.$$ The random variables under summation on the left are independent. Choosing $\mu=\mu(s)$ small enough we thus find, by \cref{lem:CL27},
$$\PP\Big(\sum_{v\in \mathrm{active}(\mathcal{E}_{K_0})}f(\Zz[v,N_\eps])<\mu S_{K_0}\, \Big|\,\mathcal{E}_{K_0}\Big) \leq\exp\Big({-\delta\frac{S^2_{K_0}}{\sum_{v\in \mathrm{active}(\mathcal{E}_{K_0})} \EE[ f(\Zz[v,N_\eps])^2|\mathcal{E}_{K_0}]} \Big)},$$
for some $\delta=\delta(\mu)>0$. Taking into account the linearisation of $f$, and \cref{prop:expbounds1}, we obtain $\EE[ f(\Zz[v,N_\eps])^2|\mathcal{E}_{K_0}]\leq C_{\ref{prop:expbounds1}}\EE f(\Zz[v,N_\eps])^2\leq C\xi(v,N_\eps)^2,$ for some constant $C>0.$ Hence $$\sum_{v\in \mathrm{active}(\mathcal{E}_{K_0})} 
\EE\big[ f(\Zz[v,N_\eps])^2\,\big|\,\mathcal{E}_{K_0} \big] \leq C \xi(\ell_{K_0},N_\eps)S_{K_0},$$ using $\sum_{i}x_i^2\leq \max |x_i|\sum_{i}x_i$ and that the maximum is attained at $\ell_{K_0}$ due to the restriction of the exploration. Therefore $$\PP\Big(\sum_{v\in \mathrm{active}(\mathcal{E}_{K_0})}f(\Zz[v,N_\eps])<\mu S_{K_0}\Big|\mathcal{E}_{K_0}\Big)=O\big(\e^{-{\log N}/{\log\log N}}\big),$$ by \eqref{eq:scorebound3} and the fact that $H_{K_0}$ is a bounded multiple of $S_{K_0}.$ Taking expectations and using the already established lower bound on the probability of a successful exploration yields the desired bound of 
$\PP(\sum_{v\in V(i)}f(\Zz[v,N_{\eps}])< \nu \xi(V(i),N_\eps))\leq {\varepsilon}/{24},$ for sufficiently large $N$.
\smallskip

Combining the distance bounds from all exploration phases and summing up all error probabilities we thus have shown that for any $\eps\in(0,1/3)$ with probability exceeding $1-3\varepsilon$, 
$$d_{N}(U,V)\leq D+2+\Big(\frac{1}{1+\alpha}+\eta\Big)\frac{\log N}{\log \log N}+2k_0(\varepsilon),$$ for all sufficiently large $N$. This concludes the proof as $\eta>0$ was arbitrary.
\end{proof}

\section{Proof of \cref{thm:CM}}
In this section we use a similar method as in the previous sections to describe the average distances in the Norros-Reittu model with i.i.d. random weights, and thus prove~\cref{thm:CM}.  The technical details are considerably easier in this case, and some parts of the proof which proceed in direct analogy to the preferential attachment case will only be sketched.
\smallskip

We first state some well known facts about heavy tailed i.i.d.\ weight sequences.
\begin{prop}[Asymptotics of weights]\label[prop]{prop:weightas}
Let $(W_i)_{i\geq 1}$ be an i.i.d. sequence satisfying \begin{equation}\label{TAILS}
\PP( W_1\geq k) =k^{-2}(\log k)^{2\alpha+o(1)},\end{equation} and denote by $F_n$ the distribution function of the $n$-th power $W_1^n$ of the weights.  For every $\varepsilon\in(0,1)$ there is a subset $\Omega_{\varepsilon}$ of the space of all infinite weight sequences with $\PP(\Omega_{\varepsilon})>1-\varepsilon$ and positive constants $C_1,C_2,C_3$ and $c_2$ 
such that on $\Omega_{\varepsilon}$ the following conditions are satisfied
\begin{equation}\label{eq:maxbound}
\max_{1\leq i\leq N}W_i\leq C_1 \big(\sfrac{1}{1-F_1}\big)^{-1}(N),
\end{equation}
\begin{equation}\label{eq:sumbound2}
c_2 \leq \frac{\sum_{i=1}^N W_i^2 - J(N)}{\big(\sfrac{1}{1-F_2}\big)^{-1}(N)} \leq C_2,
\end{equation}
\begin{equation}\label{eq:sumbound3}
\sum_{i=1}^N W_i^3 \leq C_3 \big(\sfrac{1}{1-F_3}\big)^{-1}(N),
\end{equation}
where the generalised inverse of a monotone function is chosen to be left-continuous and
$$J(N):=N\EE [W^2_1 \1\{W^2_1\leq \big(\sfrac{1}{1-F_2}\big)^{-1}(N)\}].$$
\end{prop}
\begin{proof}
Inequality \eqref{eq:maxbound} is a direct consequence of the weak convergence of the rescaled maximum weight to the Fr\'echet distribution (see e.g. \cite[Chapter I]{Re87}). The relations \eqref{eq:sumbound2} and \eqref{eq:sumbound3} follow from weak convergence of rescaled partial sums to stable random variables with positive support (see e.g. \cite[Corollary 7.1]{Re07} for a stronger functional version).
\end{proof}
\subsection{Proof of the lower bound}
It is now straightforward to deduce a first moment upper bound on the probability of existence of short paths in $\mathcal{H}_N.$
\begin{prop}[Lower bounds on distances in NR]\label[prop]{prop:CMLB}
Let $\mathcal{H}_N$ denote a Norros-Reittu network with weight distribution satisfying \eqref{eq:CMcondition}, then for every $\delta\in(0,({1+2\alpha})^{-1})$ and independently and uniformly chosen vertices $U,V\in\mathcal{H}_N$,
$$d_N(U,V)\geq\Big(\frac{1}{1+2\alpha}-\delta\Big)\frac{\log N}{\log\log N}\ \ \textrm{ with high probability as }N\to\infty.$$
\end{prop}
\begin{proof}
We use \cref{lem:1stMomentBound} conditionally on the sequence $W_1, W_2, \ldots$ of weights, and
given~$N$ we relabel the vertices of  $\mathcal{H}_N$ in decreasing order of weight and denote by
$W^{\ssup 1}\geq \cdots\geq W^{\ssup N}$ the order statistic of the first $N$ weights.
 It is sufficient to verify the conditions of the lemma,  for any $\varepsilon\in(0,1)$,  on a subset $\Omega_\varepsilon$ of the space of all  weight sequences with $\PP(\Omega_\varepsilon)\geq 1-\varepsilon$. Conditional independence of edges immediately yields \eqref{eq:cor} with $\kappa_N=1$.
Let $F$ be the distribution function of $W_1$. By \eqref{TAILS} we may fix a sequence $(\bar{\Psi}_N)_{N\in\N}$ satisfying
$\bar\Psi_N=(\log N)^{2\alpha+o(1)}$ such that, for any $\delta\in>0$ we have
$$p(v):=1-F\Big(\sqrt{\sfrac{N}{v}\bar\Psi_N}\Big)\leq \delta\frac{v}{N}, \quad \text{for all }v\in[N],$$ if $N$ is sufficiently large.
Denoting $L_N:=\sum_{n=1}^N W_n\sim N \, \EE W_1$, the conditional connection probabilities  satisfy 
\begin{equation} \label{eq:connprobb}\PP(v\leftrightarrow w)\leq\frac{W^{\ssup v}W^{\ssup w}}{L_N}.\end{equation} 
Therefore we may show that \eqref{eq:conpb} is satisfied for $\Psi_N=C(\varepsilon)^2\bar{\Psi}_N$, where $C(\varepsilon)$ is some constant such that \begin{equation} 
W^{\ssup v}\leq C(\eps)\, \sqrt{\frac{N}{v}\bar\Psi_N}\label{eq:degBound}, \ \  \textrm{for all }1\leq v\leq N, \end{equation} 
with probability exceeding $1-\varepsilon$. To demonstrate this, let $S^{\ssup{v}}_{N}$ the number of weights 
$W_1, \ldots, W_N$ exceeding \smash{$\sqrt{{(N}/{v})\bar\Psi_N}.$} 
The random variable~$S^{\ssup{v}}_{N}$ is dominated by a binomial random variable with parameters $N$ and $p(v)$,
hence  Bernstein's inequality gives, for fixed $\delta<1$,
 $$\PP(S^{\ssup{v}}_{N}>2v)\leq\exp\Big(-\frac{v^2}{2\textrm{Var}S^{\ssup{v}}_{N}+\frac{2}{3}v}\Big)\leq \e^{-\frac{3}{8}v}.$$ Let $M$ such that $\sum_{v=M}^{\infty}\e^{-3v/8}<\varepsilon/2$. Then with probability exceeding $1-\varepsilon/2$, there is no $v\geq M$ such that $W^{\ssup{2v}}>\sqrt{({N}/{v})\bar\Psi_N}$ which is equivalent to 
 \begin{equation}
 \label{evenbound}
 W^{\ssup{v}}\leq \sqrt{2}\,\sqrt{\frac{N}{v}\bar\Psi_N} \quad \text{for all even } v\geq 2M.
 \end{equation} Now if \eqref{eq:degBound} were not true for any odd index $v+1 > 2M$ and $C(\eps)>2$, this would mean in particular that $$W^{\ssup{v}}>C(\eps)\sqrt{\frac{N}{v+1}\bar\Psi_N}=C(\eps)\sqrt{\frac{v}{v+1}}\sqrt{\frac{N}{v}\bar\Psi_N}\geq \frac{C(\eps)}{\sqrt{2}}\sqrt{\frac{N}{v}\bar\Psi_N},$$ contradicting \eqref{evenbound}. We conclude that \eqref{eq:degBound} holds with $C(\eps)>2$ for all $v\geq 2M$ with probability exceeding $1-\eps/2$. Turning our attention to the weights $W^{\ssup{v}},\dots,W^{\ssup{2M}}$, we note that by a standard Poisson approximation result, see e.g.~\cite[Proposition 3.21]{Re87},
 for any $1\leq v \leq 2M$, we have that $S_N^{\ssup v}$ converges weakly to a Poisson distribution
 with parameter $\lambda:=\lim_{N\to\infty}
 N p(v)\leq 2\delta M.$
 Hence by choosing $\delta$ small enough we can ensure that, for large~$N$, we have
 $\sum_{i=1}^{2M} \PP\{ S^{\ssup i}_N > i\}\leq \eps/2,$  which completes the proof of \eqref{eq:degBound}.
Application of \cref{lem:1stMomentBound} now concludes the proof of \cref{prop:CMLB} as $\log \Psi_N=\big(2\alpha+o(1)\big)\log\log N$ and~$\kappa_N=1$.
\end{proof}

\subsection{Proof of the upper bound}\label{sec:proofsCM}
We now prove the upper bound in Theorem 2.
\begin{prop}[Upper bound on distances in NR]\label{prop:ubCM}
Let $\mathcal{H}_N$ be a Norros-Reittu network with weight distribution satisfying \eqref{eq:CMcondition}. Consider vertices $U,V$ chosen independently and  uniformly at random from the largest component $\Cc_N\subset\mathcal{H}_N$. Then, for any~$\delta>0$, $$d_N(U,V)\leq\Big(\frac{1}{1+2\alpha}+\delta\Big)\frac{\log N}{\log\log N}\ \ \textrm{ with high probability as }N\to\infty.$$
\end{prop}
This result can be obtained by a straightforward adaptation of the proof of \cite[Theorem 3.22]{hofstad_random_2012}, which uses the second moment method in combination with path counting techniques. For the closely related Chung-Lu model with deterministic weights, a related result is \cite[Theorem 7.9]{chung_complex_2006}, the proof of which also works in our setting. We provide a sketch of a proof relying on similar arguments as given in Section 4 for the preferential attachment network.\medskip

For $H\subset [N]$ we denote by $W(H)=\sum_{v\in H}W_v$ the \emph{total weight} of $H$. 
Just like in the preferential model, the neighborhood of a uniformly chosen vertex $V\in\mathcal{H}_N$ converges in distribution to a random tree $\mathfrak{S}$. This tree can be obtained by a \emph{mixed Poisson branching process}, see \cite{norros2006}.
Denoting by  $p(W)$ the probability of $\{|\mathfrak{S}|=\infty\}$, we get $\lim_{N\to\infty}{\#\mathcal C_N}/{N}=p(W)$ in probability, see \cite[Section 3.1.]{hofstad_random_2012}. 
\medskip\pagebreak[3]

The following facts are instrumental for our argument.
\begin{lem}\label[lem]{prop:local}
Choose $V\in[N]$ uniformly. For every $\varepsilon\in(0,p(W)), s_0>0$ there exists $k_0>0$, such that $\PP(W(\{v\in[N]:d_N(V,v)=k_0\})\geq s_0)\geq p(W)-\varepsilon,$ for sufficiently large~$N$.
\end{lem}
\begin{proof}
This follows from local weak convergence to $\mathfrak{S}$ and the fact that the offspring distribution of the branching process generating $\mathfrak{S}$ has infinite mean in every generation~$k\geq 2$, hence is supercritical.
\end{proof}
\begin{lem}\label[lem]{prop:coresNR}
Fix $M=\lceil\log N^R \rceil$ for some fixed $R>0$ and let $C_N$ denote the $M$ vertices with the largest weights. Then the diameter of the subgraph induced by $C_N\subset \mathcal{H}_N$ is bounded with high probability, as $N\to\infty.$
\end{lem}
\begin{proof}
Given~$N$ we relabel the vertices of  $\mathcal{H}_N$ in decreasing order of weight and denote by
$W^{\ssup 1}\geq \cdots\geq W^{\ssup N}$ the order statistics of the first $N$ weights.
Fix $\varepsilon>0$ and $\delta\in(0,\alpha)$. Then
$L_N:=\sum_{i=1}^N W_i \sim N \, \EE W_1$, and 
$$W^{\ssup v}\geq \sqrt{\frac NM}\big(\log\sfrac N M \big)^{\alpha-\delta}, \textrm{ for all } v\in[M],$$ on a subset $\Omega_{\varepsilon}$ with probability exceeding $1-\eps$, by a standard extreme value calculation, using e.g. \cite[Theorem 2.5.2]{leadbetter}. Given the weights, each pair of vertices $(v,w)\in C_N$ independently is connected with probability at least $$1-\e^{-{(W^{\ssup{M}})^2}/{L_N}}\geq \frac{\big(\log \sfrac NM\big) ^{2\alpha-2\delta}}{ 3M \EE W_1}=:p(M,N).$$ Now coupling to an Erd\H{o}s-R\'{e}nyi  graph $\Gg(M,p(N,M))$ and \cite[Corollary 10.12]{bollobs_random_2001} yield the boundedness of the diameter.
\end{proof}
\begin{lem}\label[lem]{lem:connections}
If $V_1,V_2\subset[N]$ are disjoint sets with total weights satisfying $$\lim_{N\to\infty}\frac1N W(V_1)W(V_2)=\infty \textrm{ in probability,}$$ then they are connected with high probability in $\mathcal{H}_N$.
\end{lem}
\begin{proof}
By conditional independence,  $\PP(V_1\not\leftrightarrow V_2)=\e^{-{W(V_1)W(V_2)}/{L_N}}$, from which the result follows since $W(V_1)W(V_2)/L_N$ diverges to infinity, in probability.
\end{proof}
\begin{proof}[Proof of  Proposition \ref{prop:ubCM}]
In view of \cref{prop:local,prop:coresNR,lem:connections} it is sufficient to show that a truncated exploration in $\mathcal{H}_N$ started in a configuration $\Ee_0$ of large initial weight $S_0$ with high probability, as $N\to\infty$, reaches a configuration $\Ee_k$ satisfying $$S_{k}=W(\mathrm{active}(\Ee_k))\geq \frac{\sqrt{N}}{(\log N)^{R}}$$ in less than $K$ stages, where $R, \delta>0$ are fixed and $$K=\Big(\frac{1}{2+4\alpha}+\delta \Big)\frac{\log N}{\log\log N}.$$ We truncate the exploration in the following way: at stage $k$, we only investigate connections between active vertices and vertices of weight at most $w_{k+1}$, where $(w_k)_{k\geq 1}$ is a superexponentially growing sequence specified below.
Since we would like to condition on the weights, we start by demonstrating that almost all weight sequences have certain properties. Let $(A_k)_{k=0}^{K}$ denote a partition of the set $[1,\sqrt{N}(\log N)^{2\alpha})$ into $K$ nonoverlapping intervals $A_k=[a_{k},a_{k+1})$ of equal length. Applying \cref{lem:CL27}, 
and a brief calculation we may assume that $W_1,\dots, W_N$ satisfy, 
\begin{equation}\label{eq:psumbound2nd}
\sum_{i=1}^{N}W_i^2\1\{W_i\leq w_k\}\geq \frac 12\EE \Big[\sum_{i=1}^{N}W_i^2\1\{W_i\leq w_k\}\Big],\; 
\mbox{ for }1\leq k\leq K,
\end{equation}
as well as
\begin{equation}\label{eq:psumbound3rd}
\sum_{i=1}^{N}W_i^3\1\{W_i\leq w_k\}\leq \frac 32\EE\Big[ \sum_{i=1}^{N}W_i^3\1\{W_i\leq w_k\}\Big],\; 
\mbox{ for } 1\leq k\leq K.
\end{equation}

Fix $\varepsilon>0$. Let $\mathcal{E}$ be a configuration obtained from an exploration of $\mathcal{H}_N$, $S=W(\mathrm{active}(\mathcal{E}))$, $H=\mathrm{active}(\mathcal{E})\cup\mathrm{dead}(\mathcal{E})$, $w>0$ and $V=V(w)=\{v\in\mathrm{veiled}(\mathcal{E}): W_v\leq w\}.$ It is easy to see, using an appropriate coupling to a sum of independent weighted Bernoulli random variables and \cref{lem:CL27} that, as long as $wS=o(L_N),$ \begin{equation}\label{eq:grwtbd}W(\{v\in V:v\leftrightarrow \mathrm{active}(\mathcal{E})\})\geq \frac{\sum_{v\in V}W^2_v}{4 L_N} \, S=:\nu(w,N)S,\end{equation} conditional on $\mathcal{E}$ and the weight sequence, with probability at least \begin{equation}\label{eq:errbd} 1-\e^{-\frac{(\sum_{v\in V}W_v^2)^2}{4 L_N \sum_{v\in V}W_v^3}S}.\end{equation}
Note that, by \eqref{eq:sumbound2} and our choice of weight distribution, \begin{align*}\sum_{v\in V}W_v^2
& \geq \sum_{v\in[N]}W^2_{v}\1\{W_v\leq w\} -  \big(\max_{a\in H}W_a\big) W(H).
\end{align*}
Hence choosing $w_0$ sufficiently large, setting $w_{k}=c(\delta,\varepsilon)\log N^{1+2\alpha-\eta(\delta)}w_{k-1}, \; 1\leq k\leq K,$ for some appropriately chosen small values of $c(\delta,\varepsilon),\eta(\delta)$ and letting $V_k=V(w_k)$ in \eqref{eq:grwtbd}, we obtain that the weight $S_k$ of the active vertices increases 
in each stage $k$ of the exploration by a factor of at least $\nu(w_k,N)\geq c\log(w_k)^{2\alpha +1 -\eta(\delta)},$ for some constant $c$ which depends on $\delta$ and $\varepsilon$ but not on $N$. A straightforward calculation now shows that the exploration satisfies $$S_{k}=W(\mathrm{active}(\Ee_k))\geq \frac{\sqrt{N}}{(\log N)^{R}}$$ after at most $K$ stages. Summing the error terms in \eqref{eq:errbd}
	for the different stages using \eqref{eq:psumbound2nd} and \eqref{eq:psumbound3rd}, we obtain for some constants $c_1,c_2$, which are independent of $N$,
\begin{align*}
\sum_{k=1}^K \e^{-\frac{(\sum_{v\in V_k}W_v^2)^2}{4 L_N \sum_{v\in V_k}W_v^3}S_{k-1}}\leq \sum_{k=1}^K \e^{-c_1\frac{(\log w_k)^{4\alpha+2-2\eta(\delta)}N^2}{N^2 w_k (\log w_k)^{2\alpha+\eta(\delta)}}S_{k-1}}\leq\sum_{k=1}^K \e^{-c_2(\log w_k)^{1-\eta(\delta)}}<\varepsilon,
\end{align*}
as $N\to\infty$. This concludes the proof, since $\varepsilon$ and $\delta$ where chosen arbitrarily.
\end{proof}
\appendix
\section{Further calculations for preferential attachment networks}\label{knownresults}

The following lemma is used to prove \cref{prop:localknown}. The proof  relies on a coupling of local neighbourhoods in~$\mathcal{G}_N$ with the `idealised neighbourhood tree' $\mathfrak{T}$  introduced in \cite[Section 1.3]{dereich_random_2013}, in which vertices of the tree have positions on the negative real line. We denote by $\mathfrak{T}_k$  the $k$-th generation of $\mathfrak{T}$, and by $p(f)$ be  the probability that $\mathfrak{T}$ is infinite. 

\begin{lem}\label[lemma]{prop:frakT}
Let $\chi\colon[0,\infty)\to[1,\infty)$ be a an increasing function satisfying $$c\leq {\chi(x)}{\e^{-\frac12 x}}\leq C, \textrm{ for some }0<c\leq C <\infty.$$ Denote by $\bar\chi\colon\mathfrak{T}\to[1,\infty)$ the function defined on the vertices of $\mathfrak{T}$ by $\bar{\chi}(v)=\chi(-x_v)$, where $x_v$ is the position of $v\in\mathfrak{T}$ on the negative real line.  Then, for any $s>0$, almost surely conditional on $\#\mathfrak{T}=\infty$ there exists $K\in\NN$ and  $A_K\subset\mathfrak{T}_K$ such that
$$\sum_{v\in A_K}\bar{\chi}(v)\geq s \max_{v\in A_K}\bar{\chi}(v).$$
\end{lem}

\begin{proof}
On the event $\#\mathfrak{T}=\infty$ there exists, almost surely, a sequence $(w_i)$ of vertices 
in~$\mathfrak{T}$ with positions drifting to $-\infty$, see \cite[Lemma 3.3]{dereich_random_2013}. 
We choose such a sequence adapted to the natural filtration  of the branching process.
For any~$\eta>1$,  the events that $w_i$ has a child positioned in $[-2\eta,-\eta]$  are stochastically bounded 
from below by i.i.d.\ events  of positive probability. Hence we find a vertex $v(1)$ of type $\ell$ in 
$\mathfrak{T}$ with position $x_{v(1)}\in[-2\eta,-\eta]$. Continuing inductively we construct an adapted sequence of 
vertices $v(i)$ of type~$\ell$ in~$\mathfrak{T}$ such that \smash{$x_{v(i)}\in[x_{v(i-1)}-2\eta,x_{v(i-1)}-\eta]$.} 
Denote by $A(i)$ the set of offspring generated by $v(i)$ in $[x_{v(i)},0]$ and let $Y(i)=\sum_{v\in A(i)}\bar{\chi}(v).$  By definition of the underlying branching random walk, denoting by $(Z_t)_{t\geq 0}$ the idealised degree evolution process, we have 
$$
\begin{aligned}
\EE [Y_i\, | \, x_{v(i)}=x]  
= \int_{0}^{-x}\chi(-u-x)\EE f(Z_u)\di u
\geq c\e^{\frac12 x}\int_{-{x_{v(i-1)}}}^{-x}\e^{\frac12 u}\EE f(Z_u)\di u.
\end{aligned}
$$
Using the estimate \smash{$c'u^\alpha\e^{u/2}\leq \EE f(Z_u)\leq C'(u^\alpha \vee 1)\e^{ u/2}$,} for all $u\geq 0,$ which is a continuous analogue of \cref{prop:expbounds1,prop:expbounds2} and may be shown in a similar fashion for our choice of attachment rule, we get a lower bound of
\begin{equation}\label{eq:help1}
\begin{aligned}
\EE [Y_i\, | \, x_{v(i)}=x] &  \geq cc'\, \e^{\frac12 x}\int_{-{x_{v(i-1)}}}^{-x}\e^{\frac12 u}u^\alpha \e^{\frac12 u}\di u
\geq c''\, (-x_{v(i-1)})^{\alpha}\e^{-\frac12 x}, 
\end{aligned}
\end{equation}
for some constant $c''>0$ not depending on $\eta$.
From \eqref{eq:help1} we get  $i_0(s)\in \NN$ such that 
\begin{equation}\label{eq:help2}\EE [Y_i\, |\, x_{v(i)}=x]\geq 2s\bar{\chi}(v(i)), \textrm{ for all }i\geq i_0.\end{equation}
Calculating $\EE [Y_i^2\,|\,x_{v(i)}=x]$ is slightly more subtle. We have
$$\EE\Big[\sum_{v\in A(i)}\bar{\chi}^2(v)\Big|x_{v(i)}=x\Big]\leq C'\, (-x)^{\alpha}\e^{-\frac12 x},$$ for some 
constant $C'>0$, by a calculation similar to \eqref{eq:help1}. 
Note  that, by \cite[Lemma 2.5]{dereich_random_2013}, for any $u\geq 0$, we have  
\smash{$\EE [f(Z_t)|\Delta Z_u=1]\leq \EE [f(Z_t)|Z_0=1]\leq f(1)f(0)^{-1}\EE f(Z_t),$}  for all $t\geq u.$ The offspring intensity of $v(i)$ on $[x_u,0]$ conditional on producing offspring in position $x_u$ is thus bounded by a constant multiple of the unconditional intensity. This implies that 
$$\EE\Big[\sum_{\heap{u,v\in A(i)}{u<v}}\bar{\chi}(u)\bar{\chi}(v)\,\Big|\,x_{v(i)}=x\Big]\leq C''\EE\Big[\sum_{v\in A(i)}\bar{\chi}(v)\,\Big|\,x_{v(i)}=x\Big]^2\leq C'''\,(-x)^{2\alpha}\e^{-x},\\[-2mm]$$ by a similar calculation as above. 
Combining the previous two displays gives a bound on $\EE [Y_i^2\,|\,x_{v(i)}=x]$.
Using \eqref{eq:help2} and the Paley-Zygmund inequality, we infer 
$$\PP\big(Y_i\geq s\bar{\chi}(v(i))\, \big| \, x_{v(i)}=x \big)
\geq \PP\big(Y_i\geq \sfrac12\EE[Y_i|x_{v(i)}=x]\big|x_{v(i)}=x\big)
\geq \frac{\EE[Y_i|x_{v(i)}=x]^2}{4\EE[Y_i^2|x_{v(i)}=x]}.$$ The moment estimates and assumptions on $v(i)$  imply that, for some small constants $c,q>0$, 
$$\PP(Y_i\geq s\bar{\chi}(v(i))|x_{v(i)}=x)\geq c\Big(\frac{x_{v(i-1)}}{x}\Big)^{2\alpha}\geq q>0,$$ 
as soon as $i\geq i_0.$ Clearly, $\max_{u\in A(i)}\bar{\chi}(i)$ is at most $\bar{\chi}(v(i))$, since $\bar{\chi}$ is decreasing. So each of the sets $A(i)$ has probability at least $q$ of being a set with the desired property, and the assertion follows by conditional independence of the $A(i), i\geq i_0.$
\end{proof}

\begin{proof}[Proof of \cref{prop:localknown}]
Denote the tree associated with the configuration $\mathcal{E}_k$ by $T_k.$ The arguments
of \cite{dereich_random_2013} imply that, with high probability, for any fixed $k$, the configuration
$T_k$ can be coupled to $\mathfrak{T}_k$ and the scores $\xi$ defined on $T_k$ can be associated 
to a function $\chi$ satisfying the conditions of \cref{prop:frakT} such that $\xi=\bar\chi$ on corresponding 
vertices.  The claim hence follows from \cref{prop:frakT}.
\end{proof}

\begin{lem}[\cref{lem:tediousmain}]\label[lem]{lem:tedious2}
There are $\eta\in(0,1)$ and $c>0$ only depending on $\alpha$ such that for any choice of $A\subset \{\lceil 2\e^2\rceil ,\dots, N \}$ and $v_0<\frac{\min A}{\e^2}\wedge\eta N$ satisfying
\begin{equation}\label{eq:pigcond}
\big(\log\sfrac{N}{\min A}\vee 1\big)^\alpha \xi^2(A)\leq \frac c2 N \big(\log\sfrac{N}{v_0}\big)^{\alpha+1},
\end{equation}
we have, for $V=\{v_0,\dots,N\}\setminus A$ and any $a\in A$,
\begin{equation}\label{eq:pigconcl}
\sum_{v\in V}\frac1v \big(\log\sfrac{a\vee v}{a\wedge v}\vee 1\big)^{\alpha}\geq \frac c2 \big(\log\sfrac{N}{v_0}\big)^{\alpha+1},
\end{equation}
if $N$ is suffciently large.
\end{lem}
\begin{proof}
We set $$\varepsilon_0=\e^{-(2+2(\log(\e^\alpha+{ \e^{-2}/2})))^{\frac1{1+\alpha}}},$$ $\eta=\varepsilon_0^{-2}$ and first assume that, for all $A\subset \{\lceil 2\e^2\rceil,\dots, N\}$ and $v_0<({\min A}/{\e^2})\wedge \eta N$,
\begin{equation}\label{eq:claim1}
\sum_{v=v_0}^N \frac 1v \big(\log\sfrac{a\vee v}{a\wedge v}\vee 1\big)^{\alpha}\geq \frac{1}{2^{\alpha+1}(\alpha+1)}\big(\log\sfrac{N}{v_0}\big)^{\alpha+1}\;\mbox{ for all } a\in A.
\end{equation}
Then \eqref{eq:pigcond} implies that 
\begin{align*}
\sum_{v\in V}\frac1v \big(\log\sfrac{a\vee v}{a\wedge v}\vee 1\big)^{\alpha} & \geq \sum_{v=v_0}^N \frac 1v \big(\log\sfrac{a\vee v}{a\wedge v}\vee 1\big)^{\alpha}-\sum_{v\in A} \frac 1v \big(\log\sfrac{a\vee v}{a\wedge v}\vee 1\big)^{\alpha}\\
&\geq \sfrac{1}{2^{\alpha+1}(\alpha+1)}\big(\log\sfrac{N}{v_0}\big)^{\alpha+1}-\big(\log\sfrac{N}{\min A}\vee 1\big)^{\alpha}\sum_{v\in A} \frac 1v\\
&\geq \sfrac{1}{2^{\alpha+1}(\alpha+1)}\big(\log\sfrac{N}{v_0}\big)^{\alpha+1}-c_{\eqref{eq:xitoroot}}\big(\log\sfrac{N}{\min A}\vee 1\big)^{\alpha}\,\sfrac{\xi^2(A)}{N}\\
&\geq \big(\sfrac{1}{2^{\alpha+1}(\alpha+1)}-c_{\eqref{eq:xitoroot}}\sfrac c2\big) \big(\log\sfrac{N}{v_0}\big)^{\alpha+1} =\sfrac c2\big(\log\sfrac{N}{v_0}\big)^{\alpha+1},
\end{align*}
setting $c:=\big(2^{\alpha}(1+\alpha)(1+c_{\eqref{eq:xitoroot}})\big)^{-1}$. The conclusion of the lemma holds subject to~\eqref{eq:claim1}.\smallskip

Let $a\leq \lfloor\varepsilon_0N+1\rfloor$. Observe that  $$\sum_{v=v_0}^N \frac 1v \big(\log\sfrac{a\vee v}{a\wedge v}\vee 1\big)^{\alpha}\geq \sum_{v=v_0}^{\lfloor\frac a\e\rfloor} \frac 1v \big(\log\sfrac{a}{v}\big)^{\alpha}+\sum_{v=\lceil a\e\rceil}^{N} \frac 1v \big(\log\sfrac{v}{a}\big)^{\alpha}=:\Sigma_1+\Sigma_2.$$
As $x\mapsto x^{-1}\big(\log\sfrac ax\big)^{\alpha}$ is decreasing, we find, using $v_0<{a}/{\e^2}$ in the last step, that 
\begin{align*}
\Sigma_1& \geq \int_{v_0}^{\lfloor\frac a\e\rfloor +1}\frac1x \big(\log\sfrac ax\big)^{\alpha}\di x =\sfrac{1}{1+\alpha}\Big(\big(\log\sfrac{a}{v_0}\big)^{\alpha+1}-\big(\log\sfrac{a}{\lfloor\frac a\e\rfloor +1}\big)^{\alpha+1}\Big)
\\ & \geq \sfrac{1}{1+\alpha}\Big(\big(\log\sfrac{a}{v_0}\big)^{\alpha+1}-1\Big)
\geq \sfrac{1}{2(1+\alpha)}\big(\log\sfrac{a}{v_0}\big)^{\alpha+1}.
\end{align*}
The map $x\mapsto x^{-1}\big(\log(x/a)\big)^{\alpha}$ has a unique maximum at $x=\e^\alpha a$, thus 
\begin{align}
\sum_{v=\lfloor \e^\alpha\rfloor +1}^N\sfrac1v \big(\log\sfrac va\big)^{\alpha}& \geq \int_{\lfloor \e^\alpha\rfloor +1}^{N+1}\sfrac1x \big(\log\sfrac xa\big)^{\alpha}\di x
 \geq \sfrac{1}{\alpha+1}\big(\big(\log\sfrac{N+1}{a}\big)^{\alpha+1}-\big(\log\sfrac{\lfloor \e^\alpha\rfloor +1}{a}\big)^{\alpha+1}\big)\notag\\
& \geq\sfrac{1}{\alpha+1} \big(\big(\log\sfrac{N+1}{a}\big)^{\alpha+1}-\big(\log(\e^\alpha+\sfrac1a)\big)^{\alpha+1}\big)\label{eq:sigma21}
\end{align}
and
\begin{equation}\label{eq:sigma22}
\begin{aligned}
\sum_{v=\lceil a\e \rceil}^{\lfloor\e^\alpha a\rfloor}\frac1v \big(\log\sfrac va\big)^{\alpha}& \geq \1\{\lceil a\e \rceil\leq \lfloor\e^\alpha a\rfloor\}\int_{\lceil a\e \rceil-1}^{\lfloor\e^\alpha a\rfloor}\frac1x \big(\log\sfrac xa\big)^{\alpha}\di x\\
&\geq \sfrac{\1\{\lceil a\e \rceil\leq \lfloor\e^\alpha a\rfloor\}}{\alpha+1}\big(\big(\log(\e^\alpha-\sfrac1a)\big)^{\alpha+1}-1\big)
\end{aligned}
\end{equation}
Combining \eqref{eq:sigma21} and \eqref{eq:sigma22}, we get
\begin{align*}
\Sigma_2& \geq \sfrac{1}{\alpha+1} \big(\big(\log\sfrac{N+1}{a}\big)^{\alpha+1}-\big(\log(\e^\alpha+\sfrac1a)\big)^{\alpha+1}+\1\{\lceil a\e \rceil\leq \lfloor\e^\alpha a\rfloor\}\big(\big(\log(\e^\alpha-\sfrac1a)\big)^{\alpha+1}-1\big)\big)\\
& \geq \sfrac{1}{\alpha+1} \big(\big(\log\sfrac{N+1}{a}\big)^{\alpha+1}-\big(\log(\e^\alpha+\sfrac1{ 2\e^2})\big)^{\alpha+1}-1\big)\geq  \sfrac{1}{2(\alpha+1)}\big(\log\sfrac{N+1}{a}\big)^{\alpha+1},
\end{align*}
where we used the condition $$a\leq(N+1)\exp(-(2+2(\log(\e^\alpha+\frac1{2\e^2})))^{\frac1{1+\alpha}})$$ in the last step. Combining the estimates for $\Sigma_1$ and $\Sigma_2$ yields
$$\sum_{v=v_0}^N \frac 1v \big(\log\sfrac{a\vee v}{a\wedge v}\vee 1\big)^{\alpha}\geq\sfrac{1}{2(\alpha+1)}\big(\big(\log\sfrac{a}{v_0}\big)^{\alpha+1}+\big(\log\sfrac{N+1}{a}\big)^{\alpha+1}\big)\geq \sfrac{1}{2^{\alpha+1}(\alpha+1)}\big(\log\sfrac{N}{v_0}\big)^{\alpha+1}$$
by convexity of $x\mapsto x^{\alpha+1}.$
Now consider $a\geq\lceil \varepsilon_0 N\rceil.$ We have
\begin{align*}
\sum_{v=v_0}^N\frac1v\big(1\vee\log\sfrac{a\vee v}{a\wedge v}\big)^{\alpha}& \geq \int_{v_0}^{\lceil\varepsilon_0 N\rceil}\frac1x\big(1\vee \log\sfrac ax\big)^{\alpha}\di x \geq \int_{v_0}^{\varepsilon_0 N}\frac1x\big(\log\sfrac {\varepsilon_0 N}x\big)^{\alpha}\di x \\
& =\sfrac{1}{\alpha+1}\big(\log\sfrac{\varepsilon_0 N}x\big)^{\alpha+1}.
\end{align*}
Since $$\Big(\log\frac{\varepsilon_0 N}x\Big)^{\alpha+1}\geq \sfrac{1}{K}(\log\sfrac{N}{v_0})^{\alpha+1}$$
if and only if  $$v_0\leq N\varepsilon_0^{({1-(\frac{1}{K})^{\frac{1}{1+\alpha}}})^{-1}},$$ 
we choose $K=2^{\alpha+1}$ and the desired bound~\eqref{eq:claim1} follows.
\end{proof}
\begin{lem}[\cref{lem:elldecay}]\label[lem]{lem:tedious3}
For any $\alpha\geq0,\delta\in(0,2\alpha+2)$ let \begin{equation}\label{eq:kOdef}k_0(\delta,\alpha)=\min\{k \geq 3 :\delta \log k \geq (2\alpha+2-\delta)k\log(1+\sfrac1k)+1 \}\end{equation}
then $$\ell_k\leq N\e^{-(2\alpha+2-\delta)(k-k_0)\log k}\;\textrm{ for all } k_0\leq k <K_*(N).$$ Furthermore, there is a constant $c>0$ depending only on $s_0$ such that $$\ell_k\geq CN\e^{-(4\alpha+5)k(1\vee\log k)},\; \textrm{ for all }k.$$
\end{lem}
\begin{proof}
We first show the upper bound by induction in $k$. For $k=k_0$ the assertion is trivially true as soon as $N$ is large enough. Now assume  that \smash{$\ell_k\leq N\e^{-(2\alpha+2-\delta)(k-k_0)\log k}$} for some $k<K_*(N)-1$ then we have, by definition of $(\ell_k)_{k\geq 1},$ that
$\log\ell_{k+1}\leq \log \ell_k -(2\alpha+2)\log(\log N-\log\ell_k)+1$ 
and applying the induction hypothesis yields
\begin{align*}
\log\sfrac{\ell_{k+1}}{N}\leq -(2\alpha+2-\delta)(k-k_0)\log(k+1) &+\big((2\alpha+2-\delta)(k-k_0)\log\sfrac{k+1}{k}+1\big)\\
&-(2\alpha+2)\log\big((k+1)\sfrac{k}{k+1}(2\alpha+2)\log k\big).
\end{align*}
By \eqref{eq:kOdef} we have $\sfrac{k}{k+1}(2\alpha+2)\log k\geq 1$, hence
\begin{align*}
\log\sfrac{\ell_{k+1}}{N}\leq &-(2\alpha+2-\delta)(k+1-k_0)\log(k+1)\\
 &+\big((2\alpha+2-\delta)(k-k_0)\log\sfrac{k+1}{k}+1-\delta\log(k+1)\big).
\end{align*}
The second term of the sum is negative by \eqref{eq:kOdef} and the induction is complete.
The lower bound follows by a similar argument.
\end{proof}

\begin{ack}
We would like to thank the anonymous referee for a very careful reading of the article and numerous suggestions which helped to improve the presentation of our results. 
\end{ack}

\end{document}